\documentclass[eqthmnum,nocolour]{jt-calcs}
\usepackage[backend=bibtex,style=alphabetic,sorting=nyt,maxbibnames=99]{biblatex}
\bibliography{bibliography}

\usepackage{mathrsfs}
\usepackage{tikz-cd}

\usepackage{tikz}
\usetikzlibrary{matrix}

\setlist[enumerate,1]{label=(\alph*)}

\DeclareMathOperator{\der}{der}

\DeclareMathOperator{\Tor}{Tor}

\title{The Structure of Root Data and Smooth Regular Embeddings of Reductive Groups}
\author{Jay Taylor}
\mscno{2010}{20G07}{20C33}

\keywords{Reductive groups, root data, smooth regular embeddings}

\begin{document}
\begin{abstract}
We investigate the structure of root data by considering their decomposition as a product of a semisimple root datum and a torus. Using this decomposition we obtain a parameterisation of the isomorphism classes of all root data. By working at the level of root data we introduce the notion of a smooth regular embedding of a connected reductive algebraic group, which is a refinement of the commonly used regular embeddings introduced by Lusztig. In the absence of Steinberg endomorphisms such embeddings were constructed by Benjamin Martin.

\indent In an unpublished manuscript Asai proved three key reduction techniques that are used for reducing statements about arbitrary connected reductive algebraic groups, equipped with a Frobenius endomorphism, to those whose derived subgroup is simple and simply connected. By using our investigations into root data we give new proofs of Asai's results and generalise them so that they are compatible with Steinberg endomorphisms. As an illustration of these ideas, we answer a question posed to us by Olivier Dudas concerning unipotent supports.
\end{abstract}

\section{Introduction}
\begin{pa}
Let $\mathbb{K}$ be an algebraically closed field of characteristic $p \geqslant 0$ then we say $\bG$ is a $\mathbb{K}$-group if it is an algebraic group defined over $\mathbb{K}$. We will denote by $\mathscr{G}_{\mathbb{K}}$ the set of isomorphism classes of connected reductive $\mathbb{K}$-groups. In this article we will be concerned with the following natural question.
\end{pa}

\begin{prob}\label{prob:iso-alg-grps}
Give a nice parameterisation of the set of isomorphism classes $\mathscr{G}_{\mathbb{K}}$.
\end{prob}

\begin{pa}\label{pa:subtleties}
Here the term nice is subjective. However, for each element of $\mathscr{G}_{\mathbb{K}}$ we would like to give it a computable label which distinguishes it uniquely. If $p > 0$ then it is difficult to approach \cref{prob:iso-alg-grps} in the language of algebraic groups. This is primarily because a bijective morphism of algebraic groups need not be an isomorphism in positive characteristic. For instance, if $p > 0$, then the product map $\SL_p(\mathbb{K}) \times \mathbb{G}_m \to \GL_p(\mathbb{K})$, where $\mathbb{G}_m$ denotes the multiplicative group of the field, is a bijective morphism of algebraic groups but the $\mathbb{K}$-groups $\SL_p(\mathbb{K}) \times \mathbb{G}_m$ and $\GL_p(\mathbb{K})$ are not isomorphic.
\end{pa}

\begin{pa}
To get around these subtleties we will provide an answer to \cref{prob:iso-alg-grps} using the language of root data. If $V$ is a set then we denote by $\mathbb{Q}V = \mathbb{Q} \otimes_{\mathbb{Z}} \mathbb{Z}V$ the $\mathbb{Q}$-vector space obtained from the free $\mathbb{Z}$-module $\mathbb{Z}V$ by extending scalars. With this we recall that, roughly, a root datum is a quadruple $(X,\Phi,\widecheck{X},\widecheck{\Phi})$ where $X$ and $\widecheck{X}$ are a pair of finite rank free $\mathbb{Z}$-modules in duality and $\Phi \subseteq X \subseteq \mathbb{Q}X$ and $\widecheck{\Phi} \subseteq \widecheck{X} \subseteq \mathbb{Q}\widecheck{X}$ are root systems. Let $\mathscr{R}$ denote the set of isomorphism classes of root data. By a classical theorem of Chevalley we have a bijective map $\mathscr{G}_{\mathbb{K}} \to \mathscr{R}$ defined by $\bG \mapsto \mathcal{R}(\bG)$, where $\mathcal{R}(\bG)$ is the root datum of $\bG$ defined with respect to some (any) maximal torus of $\bG$. Hence, \cref{prob:iso-alg-grps} is equivalent to the following.
\end{pa}

\begin{prob}\label{prob:iso-root-data}
Give a nice parameterisation of the set of isomorphism classes $\mathscr{R}$.
\end{prob}

\begin{rem}
To get around the subtleties in \cref{pa:subtleties} one could also work with group schemes. However, by working with root data we are able to obtain qualitative statements, such as \cref{prop:aut-ab-lift}. We also hope our work here will have applications to the implementation of algebraic groups, via root data, in computer algebra systems such as GAP, see \cite{geck-hiss:1996:CHEVIE,michel:2015:the-development-version-of-CHEVIE}.
\end{rem}

\begin{pa}
The first main goal of this paper is to give a complete solution to \cref{prob:iso-root-data}. Let us outline our approach to this problem by describing our strategy in the language of groups. For the purposes of this description we will ignore the subtleties described above; the reader can assume $\mathbb{K} = \mathbb{C}$ if they wish. It is well known that each connected reductive $\mathbb{K}$-group $\bG$ is a product $\bG_{\der}Z^{\circ}(\bG)$ of its derived subgroup $\bG_{\der} \leqslant \bG$, which is a semisimple group, and connected centre $Z^{\circ}(\bG) \leqslant \bG$, which is a torus. Another way of saying this is that the natural product homomorphism $\bG_{\der} \times Z^{\circ}(\bG) \to \bG$ is surjective. The kernel of this homomorphism is given by the subgroup $K(\bG) := \{(g,g^{-1}) \mid g \in \bG_{\der} \cap Z^{\circ}(\bG)\}$, which is a finite abelian group.
\end{pa}

\begin{pa}
If $\phi : \bG \to \bG'$ is an isomorphism then certainly this restricts to isomorphisms $\bG_{\der} \to \bG_{\der}'$, $Z^{\circ}(\bG) \to Z^{\circ}(\bG')$, and $K(\bG) \to K(\bG')$. Hence, a starting point for the classification is to classify, up to isomorphism, all triples $(\bG,\bT,K)$ consisting of: a semisimple group $\bG$, a torus $\bT$, and a finite abelian group $K$. However each such class of groups has a well-known classification result. For instance: semisimple groups are classified by a root system and a subgroup of its fundamental group, tori are classified by their dimension, and finite abelian groups are classified by their invariant factors.
\end{pa}

\begin{pa}
Now one can define various closed embeddings $\pi : K \to \bG \times \bT$ of the finite abelian group $K$ such that the quotient $(\bG\times\bT)/\pi(K)$ is a connected reductive algebraic group. To finish the classification one needs to parameterise those embeddings $\pi$ giving rise to non-isomorphic connected reductive algebraic groups. As one would expect the automorphism group of $K$ plays a role here. One of the main results of this paper, see \cref{thm:smooth-reg-embed-classification}, effectively gives a parameterisation of the groups $(\bG\times\bT)/\pi(K)$ in terms of $\Aut(K)$.
\end{pa}

\begin{pa}
Once one has obtained such a parameterisation it is possible to ask the following natural question: How many groups, up to isomorphism, give rise to a fixed triple $(\bG,\bT,K)$, taken up to isomorphism? It is, perhaps, not so surprising that when the dimension of $\bT$ is large there is only one group, up to isomorphism, yielding the triple $(\bG,\bT,K)$. However, using our parameterisation, together with work of Diaconis--Graham \cite{diaconis-graham:1999:the-graph-of-generating-sets}, we see that ``large'' here is in fact surprisingly small, see \cref{prop:aut-ab-lift}.
\end{pa}

\subsection{Smooth Regular Embeddings}
\begin{pa}
Let us now discuss the second main focus of this article, namely regular embeddings, but, first, some notation. If $A$ is a finitely generated $\mathbb{Z}$-module then we denote by $\Tor(A) \leqslant A$ the torsion subgroup of $A$ and we denote by $\Tor_p(A) \leqslant \Tor(A)$ its $p$-torsion subgroup. We then set $\Tor_{p'}(A) = \Tor(A)/\Tor_p(A)$ and $A_{p'} = A/\Tor_p(A)$. Note that $\Tor_{p'}(A) = \Tor(A_{p'})$ is the torsion subgroup of $A_{p'}$ so $A_{p'}$ has no $p$-torsion.
\end{pa}

\begin{pa}
Assume $\bG$ is a connected reductive $\mathbb{K}$-group with root datum $\mathcal{R}(\bG) = (X,\Phi,\widecheck{X},\widecheck{\Phi})$. By \cite[4.1]{bonnafe:2006:sln} we have $X(Z(\bG)/Z^{\circ}(\bG)) \cong \Tor_{p'}(X/\mathbb{Z}\Phi)$ so $Z(\bG)$ is connected if and only if $\Tor_{p'}(X/\mathbb{Z}\Phi) = \{0\}$ or, equivalently, $\Tor(X/\mathbb{Z}\Phi) = \Tor_p(X/\mathbb{Z}\Phi)$. We say $\bG$ has a connected and \emph{smooth} centre if $\Tor(X/\mathbb{Z}\Phi) = \{0\}$. This is equivalent to the condition that the unique connected reductive $\mathbb{C}$-group $\bG_{\mathbb{C}}$ satisfying $\mathcal{R}(\bG_{\mathbb{C}}) \cong \mathcal{R}(\bG)$ has a connected centre. Alternatively, it is equivalent to the scheme theoretic centre of $\bG$ (viewed as a group scheme) being connected and smooth.
\end{pa}

\begin{pa}
A closed embedding $\iota : \bG \to \widetilde{\bG}$ between $\mathbb{K}$-groups is said to be a derived embedding if $\iota(\bG_{\der}) = \widetilde{\bG}_{\der}$. Following Lusztig \cite[7]{lusztig:1988:reductive-groups-with-a-disconnected-centre} we say a derived embedding $\iota : \bG \to \widetilde{\bG}$ is a \emph{regular embedding} if $\widetilde{\bG}$ is a connected reductive $\mathbb{K}$-group with a connected centre. In this paper we call a regular embedding $\iota : \bG \to \widetilde{\bG}$ a \emph{smooth regular embedding} if $Z(\widetilde{\bG})$ is connected and smooth. If $p = 0$ then there is no difference between a regular embedding and a smooth regular embedding. However, if $p > 0$ then smooth regular embeddings are somewhat better behaved. For instance, we have a natural closed embedding $\SL_p(\mathbb{K}) \to \SL_p(\mathbb{K}) \times \mathbb{G}_m$ which is a regular embedding but not a smooth regular embedding. However the natural embedding $\SL_n(\mathbb{K}) \to \GL_n(\mathbb{K})$ is always a smooth regular embedding.
\end{pa}

\begin{pa}
We will be particularly interested in these notions when $\bG$ is equipped with a Steinberg/Frobenius endomorphism $F : \bG \to \bG$. In this case we tacitly assume that $\mathbb{K} = \overline{\mathbb{F}_p}$ is an algebraic closure of the finite field $\mathbb{F}_p$ of prime order $p>0$. Note, we take the term Steinberg endomorphism to mean that some power of $F$ is a Frobenius endomorphism. If $\bG$ is equipped with a Steinberg/Frobenius endomorphism $F : \bG \to \bG$ then given any derived, regular, or smooth regular embedding $\iota : \bG \to \widetilde{\bG}$ we require that $\widetilde{\bG}$ is equipped with a Steinberg/Frobenius endomorphism $F : \widetilde{\bG} \to \widetilde{\bG}$ such that $\iota\circ F = F \circ \iota$. This additional condition complicates matters considerably.
\end{pa}

\begin{pa}
If $\bG$ is a connected reductive $\mathbb{K}$-group, possibly equipped with a Steinberg/Frobenius endomorphism, then Deligne--Lusztig constructed a regular embedding $\iota : \bG \to \widetilde{\bG}$ in \cite[5.18]{deligne-lusztig:1976:representations-of-reductive-groups}. Hence regular embeddings always exist. By mimicking their construction at the level of root data we are able to construct a smooth regular embedding $\iota : \bG \to \widetilde{\bG}$ thus showing that smooth regular embeddings always exist, see \cref{lem:smooth-reg-embed}. In \cref{lem:no-smooth-embedding-suzuki} we give an example of the complexities that can be introduced by Steinberg endomorphisms.
\end{pa}

\begin{rem}
It was kindly pointed out to us by the referee that our construction has appeared before in the work of Benjamin Martin, see \cite[Theorem 4.5]{martin:1999:etale-slices-for-representation-varieties}. Our contribution is to note that this construction is compatible with Steinberg/Frobenius endomorphisms.
\end{rem}

\subsection{Asai's Reduction Techniques}
\begin{pa}
In an unpublished manuscript \cite{asai:1985:endomorphism-algebras} Asai introduced several reduction techniques for working with connected reductive $\mathbb{K}$-groups equipped with Frobenius endomorphisms. These techniques are used to show that a statement concerning an arbitrary connected reductive $\mathbb{K}$-group $\bG$ equipped with a Frobenius endomorphism holds by reducing to the case where the derived subgroup of $\bG$ is simple and simply connected. These ideas have been used extensively throughout the literature. For instance, they are used by: Lusztig \cite{lusztig:1984:characters-of-reductive-groups,lusztig:1988:reductive-groups-with-a-disconnected-centre} in the classification of irreducible characters, Geck \cite{geck:1996:on-the-average-values} to prove the existence of unipotent supports, and the author \cite{taylor:2016:GGGRs-small-characteristics} to prove the existence of wave-front sets.
\end{pa}

\begin{pa}
Our main purpose for introducing smooth regular embeddings is to give new proofs for Asai's reduction techniques. One significant up-shot of our proofs is that they allow us to obtain these statements for Steinberg endomorphisms and not just for Frobenius endomorphisms. We note that in the absence of Steinberg/Frobenius endomorphisms these results are much easier to obtain. Our first main result in this direction is the following, which is a strengthened form of \cite[2.3.2]{asai:1985:endomorphism-algebras}.
\end{pa}

\begin{thm}\label{lem:completion-of-embedding}
Assume $\bG$ is a connected reductive $\mathbb{K}$-group and $\sigma_i : \bG \to \bG_i$ are derived embeddings with $i \in \{1,2\}$. Then there exists a connected reductive $\mathbb{K}$-group $\bG'$ and smooth regular embeddings $\sigma_i' : \bG_i \hookrightarrow \bG'$, with $i \in \{1,2\}$, such that the following diagram commutes
\begin{equation*}
\begin{tikzpicture}
\matrix (m) [matrix of math nodes, row sep=3em, column sep=3em, text height=1.5ex, text depth=0.25ex]
{ \bG & \bG_1\\
  \bG_2 & \bG'\\ };

\path [>=stealth,->]	(m-1-1) edge node[font=\small,above,text height=1.5ex,text depth=0.25ex,->] {$\sigma_1$} (m-1-2)
		(m-1-1) edge node[font=\small,left,text height=1.5ex,text depth=0.25ex] {$\sigma_2$} (m-2-1)
		(m-1-2) edge node[font=\small,right,text height=1.5ex,text depth=0.25ex] {$\sigma_1'$} (m-2-2)
		(m-2-1) edge node[font=\small,below,text height=1.5ex,text depth=0.25ex] {$\sigma_2'$} (m-2-2);
\end{tikzpicture}
\end{equation*}
\end{thm}

\begin{pa}\label{pa:central-morphism}\label{def:smooth-cover}
To state the second reduction technique we will need to introduce a notion which is dual to that of a smooth regular embedding. We say a homomorphism of $\mathbb{K}$-groups $\pi : \widetilde{\bG} \to \bG$ is a \emph{smooth covering} if it is a surjective central homomorphism such that: $\widetilde{\bG}$ is a connected reductive $\mathbb{K}$-group, $\Ker(\pi)$ is a torus, the derived subgroup of $\widetilde{\bG}$ is simply connected, and $Z(\widetilde{\bG})$ is connected/smooth if $Z(\bG)$ is. As before, if $\bG$ is endowed with a Steinberg/Frobenius endomorphism $F : \bG \to \bG$ then we additionally require that $\widetilde{\bG}$ is endowed with a Steinberg/Frobenius endomorphism $F : \widetilde{\bG} \to \widetilde{\bG}$ such that $\pi\circ F = F\circ\pi$. The following is a strengthened form of \cite[2.3.1]{asai:1985:endomorphism-algebras}.
\end{pa}

\begin{thm}\label{prop:quotient-by-sc-der}
For any connected reductive $\mathbb{K}$-group $\bG$ there exists a smooth covering $\pi : \widetilde{\bG} \to \bG$.
\end{thm}

\begin{rem}
Our proof of \cref{prop:quotient-by-sc-der} uses duality together with the existence of smooth regular embeddings. We note that the idea of using duality and regular embeddings appears in \cite[1.7.13]{geck-malle:2016:reductive-groups-and-steinberg-maps}.
\end{rem}

\begin{pa}\label{pa:permuted-factors}
To state the third, and final, reduction technique we assume that our connected reductive $\mathbb{K}$-group $\bG$ is a direct product $\bG_1 \times \cdots \times \bG_n$ and $F : \bG \to \bG$ is a Steinberg endomorphism such that $F(\bG_i) = \bG_{i+1}$ with the indices computed cyclically. In particular, this implies that $\bG_i$ is abstractly isomorphic to $\bG_{i+1}$ and we have $F^n(\bG_i) = \bG_i$ for all $1 \leqslant i \leqslant n$ so $F^n$ restricts to a Steinberg endomorphism of $\bG_i$. Now we clearly have a natural surjective homomorphism of algebraic groups $\pi_1 : \bG \to \bG_1$ given by projection onto the first factor. The following is a strengthened form of \cite[2.4.2]{asai:1985:endomorphism-algebras}.
\end{pa}

\begin{lem}\label{lem:cyc-perm-smooth-embed}
Let $\bG$ and $F : \bG \to \bG$ be as in \cref{pa:permuted-factors}. There exists a smooth regular embedding $\sigma : \bG \to \bG'$ such that $\bG' = \bG_1' \times \cdots \times \bG_n'$ and $F' : \bG' \to \bG'$ is a Steinberg endomorphism satisfying $F'(\bG_i') = \bG_{i+1}'$ with the indices computed cyclically. Moreover, if $\pi_1 : \bG \to \bG_1$ and $\pi_1' : \bG' \to \bG_1'$ are the natural projection maps then the following diagram is commutative
\begin{equation*}
\begin{tikzpicture}
\matrix (m) [matrix of math nodes, row sep=3em, column sep=3em, text height=1.5ex, text depth=0.25ex]
{ \bG & \bG'\\
  \bG_1 & \bG_1'\\ };

\path [>=stealth,->]	(m-1-1) edge node[font=\small,above,text height=1.5ex,text depth=0.25ex,->] {$\sigma$} (m-1-2)
		(m-1-1) edge node[font=\small,left,text height=1.5ex,text depth=0.25ex] {$\pi_1$} (m-2-1)
		(m-1-2) edge node[font=\small,right,text height=1.5ex,text depth=0.25ex] {$\pi_1'$} (m-2-2)
		(m-2-1) edge node[font=\small,below,text height=1.5ex,text depth=0.25ex] {$\sigma$} (m-2-2);
\end{tikzpicture}
\end{equation*}
and the restriction $\sigma|_{\bG_1} : \bG_1 \to \bG_1'$ is a smooth regular embedding with respect to the Steinberg endomorphism $F^n$ on $\bG_1$.
\end{lem}

\subsection{Unipotent Supports}
\begin{pa}\label{pa:uni-supp-setup}
Finally, as an illustration of these reduction techniques we answer a question posed to us by Olivier Dudas concerning unipotent supports. For this, assume $\bG$ is a connected reductive algebraic $\mathbb{K}$-group equipped with a Steinberg endomorphism $F : \bG \to \bG$. Let $\mathcal{O} \subseteq \bG$ be an $F$-stable unipotent class and let us choose a set of representatives $u_1,\dots,u_r \in \mathcal{O}^F$ for the $\bG^F$-classes contained in $\mathcal{O}^F$. If $\chi \in \Irr(\bG^F)$ is an irreducible character then we define the average value of $\chi$ on $\mathcal{O}$ to be
\begin{equation*}
\AV(\mathcal{O},\chi) = \sum_{i=1}^r [A_{\bG}(u_i):A_{\bG}(u_i)^F]\chi(u_i).
\end{equation*}
If $g \in \bG$ then $A_{\bG}(g) = C_{\bG}(g)/C_{\bG}^{\circ}(g)$ denotes the component group of the centraliser. Now, given $\chi \in \Irr(\bG^F)$ one has a corresponding unipotent support $\mathcal{O}_{\chi} \subseteq \bG$, see \cite{lusztig:1992:a-unipotent-support,geck-malle:2000:existence-of-a-unipotent-support}. This is defined to be the unique class of maximal dimension satisfying the condition $\AV(\mathcal{O},\chi) \neq 0$.
\end{pa}

\begin{thm}\label{thm:uni-supp-non-vanishing}
Assume $p$ is a good prime for $\bG$ and $Z(\bG)$ is connected. Then for any $\chi \in \Irr(\bG^F)$ we have $\mathcal{O}_{\chi}$ is the unique class satisfying the following conditions:
\begin{enumerate}
	\item $\chi(u) \neq 0$ for some $u \in \mathcal{O}_{\chi}^F$,
	\item if $g \in \bG^F$ is an element such that $\chi(g) \neq 0$ then $g_{\uni} \in \overline{\mathcal{O}_{\chi}}$.
\end{enumerate}
Here $g_{\uni}$ denotes the unipotent part of $g$.
\end{thm}

\begin{pa}
The unicity of a class satisfying the conditions of \cref{thm:uni-supp-non-vanishing} is easy. Indeed, assume $\mathcal{O}$ also satisfies the conditions of \cref{thm:uni-supp-non-vanishing} then we must have $\mathcal{O} \subseteq \overline{\mathcal{O}_{\chi}}$ and $\mathcal{O}_{\chi} \subseteq \overline{\mathcal{O}}$. However this implies $\overline{\mathcal{O}} = \overline{\mathcal{O}_{\chi}}$ so $\mathcal{O} = \mathcal{O}_{\chi}$. It thus suffices to show that $\mathcal{O}_{\chi}$ satisfies these two conditions. However (a) obviously holds because $\AV(\mathcal{O}_{\chi},\chi) \neq 0$ so it suffices to show (b) holds; this is done in \cref{thm:unip-supp} using the results in \cite{achar-aubert:2007:supports-unipotents-de-faisceaux,lusztig:1992:a-unipotent-support,taylor:2016:GGGRs-small-characteristics}.
\end{pa}

\begin{pa}
Let us finally consider the assumptions in \cref{thm:uni-supp-non-vanishing}. It is known that the conclusion of \cref{thm:uni-supp-non-vanishing} is false if $p$ is a bad prime, as is pointed out by Geck in \cite[\S1]{geck:1996:on-the-average-values}. Having said this we conjecture that the conclusion of \cref{thm:uni-supp-non-vanishing} holds assuming only that $p$ is a good prime; so allowing $Z(\bG)$ to be disconnected. If $p$ is an acceptable prime (which implies $p$ is good) and $q$ is large enough then the results of \cite{lusztig:1992:a-unipotent-support,taylor:2016:GGGRs-small-characteristics} show that the conclusion of \cref{thm:uni-supp-non-vanishing} holds even when $Z(\bG)$ is disconnected. Unfortunately regular embeddings seem to be of little use in generalising \cref{thm:uni-supp-non-vanishing} to the disconnected centre case.
\end{pa}

\subsection{Outline of the Paper}
\begin{pa}
We first recall, in \cref{sec:presentation-of-Z-mod,sec:root-data}, some standard results and constructions concerning $\mathbb{Z}$-modules and root data that will be used throughout the paper. In \cref{sec:central-products} we introduce the notion of a central product of root data, which is the key construction that is to be used throughout. We prove our main classification result, namely \cref{thm:smooth-reg-embed-classification}, in \cref{sec:struct-class-root-data}. In \cref{sec:smooth-reg-embed} we investigate smooth regular embeddings and prove their existence.
\end{pa}

\begin{pa}
In \cref{sec:asai-red-technique,sec:cyc-perm-factors} we prove our strengthened forms of Asai's reduction techniques, namely \cref{lem:completion-of-embedding,prop:quotient-by-sc-der,lem:cyc-perm-smooth-embed}. In the setting of \cref{pa:permuted-factors} we also consider the relationship between Deligne--Lusztig induction and restriction on $\bG_1^{F^n}$ and $\bG^F$, see \cref{prop:comm-DL-ind}, and consequently Lusztig series, see \cref{cor:L-series-cyc-perm}. In the final section, \cref{sec:unip-supp}, we prove the necessary statements needed to obtain \cref{thm:uni-supp-non-vanishing}.
\end{pa}

\begin{acknowledgments}
The author gratefully acknowledges the financial support of INdAM and the European Commission via an INdAM Marie Curie Fellowship as well as the University of Padova via grants CPDA125818/12 and 60A01-4222/15. He would also like to thank: Meinolf Geck for several useful discussions, Gunter Malle for his comments on a preliminary version of this article, and Alan Logan for pointing us towards \cite{diaconis-graham:1999:the-graph-of-generating-sets}. Finally he thanks an anonymous referee for pointing out \cite{martin:1999:etale-slices-for-representation-varieties}.
\end{acknowledgments}

\section{Presentations of Finite \texorpdfstring{$\mathbb{Z}$}{Z}-Modules}\label{sec:presentation-of-Z-mod}
\begin{assumption}
Throughout all $\mathbb{Z}$-modules are assumed to be finitely generated.
\end{assumption}

\begin{pa}\label{pa:inv-factors}
Let $A \neq \{0\}$ be a finite $\mathbb{Z}$-module, i.e., a finite abelian group, then by \cite[II, 2.6(ii)]{hungerford:1980:algebra} there exists a unique sequence of integers $(d_1,\dots,d_s)$ such that
\begin{equation*}
A \cong \mathbb{Z}/d_1\mathbb{Z} \oplus \cdots \oplus \mathbb{Z}/d_s\mathbb{Z},
\end{equation*}
$d_s > 1$, and $d_s \mid d_{s-1} \mid \cdots \mid d_1$. The integers occurring in the sequence $(d_1,\dots,d_s)$ are called the \emph{invariant factors} of $A$. We define the invariant factors of $A = \{0\}$ to be the empty sequence $()$. In particular, $\{0\}$ has $0$ invariant factors. If $X$ is a free $\mathbb{Z}$-module then by \cite[II, 2.6(i)]{hungerford:1980:algebra} there is a unique integer $r \geqslant 0$ such that $X \cong \mathbb{Z}^r$; we call $\rk(X) := r$ the \emph{rank} of $X$. The following is an obvious consequence of \cite[II, 1.6]{hungerford:1980:algebra}.
\end{pa}

\begin{lem}\label{lem:adapted-basis}
Assume $f : X \to A$ is a surjective $\mathbb{Z}$-module homomorphism where $X$ is free and $A$ is finite. There exists a basis $(x_1,\dots,x_n)$ of $X$ and integers $d_n \mid d_{n-1} \mid \cdots \mid d_1$ such that $(d_1x_1,\dots,d_nx_n)$ is a basis of $\Ker(f)$, $(d_1,\dots,d_s)$ are the invariant factors of $A$, and $d_i = 1$ for any $i > s \geqslant 0$. In particular, we have $\rk(X) \geqslant s$.
\end{lem}

\begin{definition}
If $f : X \to A$ is a surjective homomorphism with $X$ free and $A$ finite then we say a basis $(x_1,\dots,x_n)$ of $X$ is \emph{adapted} to $f$ if it satisfies the properties of \cref{lem:adapted-basis}.
\end{definition}

\section{Root Data}\label{sec:root-data}
\begin{assumption}
Throughout $\bG$ will denote a connected reductive $\mathbb{K}$-group. Moreover, we denote by $\mathbb{G}_a$, resp., $\mathbb{G}_m$, the set $\mathbb{K}$, resp., $\mathbb{K}\setminus\{0\}$, viewed as an algebraic group under addition, resp., multiplication.
\end{assumption}

\begin{pa}\label{pa:root-data-morphism}
Recall that a root datum is a quadruple $\mathcal{R} := (X,\Phi,\widecheck{X},\widecheck{\Phi})$ such that $X$ and $\widecheck{X}$ are free $\mathbb{Z}$-modules equipped with a perfect pairing $\langle -,-\rangle_{\mathcal{R}} : X \times \widecheck{X} \to \mathbb{Z}$, and $\Phi \subseteq X$ and $\widecheck{\Phi} \subseteq \widecheck{X}$ are finite subsets satisfying the conditions of \cite[7.4.1]{springer:2009:linear-algebraic-groups}; our root data are thus assumed to be reduced. In particular, we have a bijection $\Phi \to \widecheck{\Phi}$ which we denote by $\alpha \mapsto \widecheck{\alpha}$. If $\Phi$, or equivalently $\widecheck{\Phi}$, is empty then we say $\mathcal{R}$ is a \emph{torus}. If $\mathbb{Q}\Phi = \mathbb{Q}X$, or equivalently $\mathbb{Q}\widecheck{\Phi} = \mathbb{Q}\widecheck{X}$, then we say $\mathcal{R}$ is \emph{semisimple}. Note also that the quadruple $\widecheck{\mathcal{R}} := (\widecheck{X},\widecheck{\Phi},X,\Phi)$ is again a root datum with the canonical pairing which we call the \emph{dual root datum}. If $\bT \leqslant \bG$ is a maximal torus of our connected reductive algebraic group then we may construct the root datum $\mathcal{R}(\bG,\bT) = (X(\bT),\Phi,\widecheck{X}(\bT),\widecheck{\Phi})$ of $\bG$ with respect to $\bT$, or simply the root datum of $(\bG,\bT)$, which is defined as in \cite[7.4.3]{springer:2009:linear-algebraic-groups}. In particular, we have $X(\bT)$, resp., $\widecheck{X}(\bT)$, is the character, resp., cocharacter, group of $\bT$.
\end{pa}

\begin{pa}\label{pa:p-morphism}
If $\mathcal{R} = (X,\Phi,\widecheck{X},\widecheck{\Phi})$ and $\mathcal{R}' = (X',\Phi',\widecheck{X}',\widecheck{\Phi}')$ are root data then we say $(f,q,\tau) : \mathcal{R}' \to \mathcal{R}$ is a \emph{$p$-morphism} if $f : X' \to X$ is a $\mathbb{Z}$-module homomorphism, $q : \Phi \to \{\max(p^n,1) \mid n\geqslant 0$ an integer$\}$ is a function, and $\tau : \Phi \to \Phi'$ is a bijection satisfying the condition of \cite[9.6.3, Eq.\ (44)]{springer:2009:linear-algebraic-groups}. A $p$-morphism is a $p$-isogeny if $f$ and its dual, or transpose, $\widecheck{f} : \widecheck{X} \to \widecheck{X}'$ with respect to $\langle-,-\rangle_{\mathcal{R}}$ are injective. If $\mathcal{R} = \mathcal{R}'$ and $p>0$ then we have the notion of a $p$-Steinberg, resp., $p$-Frobenius, endomorphism which are those $p$-morphisms satisfying the condition of \cite[1.4.17(ii)]{geck-malle:2016:reductive-groups-and-steinberg-maps}, resp.,  \cite[1.4.27(ii)]{geck-malle:2016:reductive-groups-and-steinberg-maps}.
\end{pa}

\begin{rem}\label{pa:dual-p-morphism}
We note that there is a natural composition of $p$-morphisms. Moreover, given a $p$-morphism $(f,q,\tau)$ we have $\widecheck{f} : \widecheck{X} \to \widecheck{X}'$ is naturally a $p$-morphism and the assignment $f \mapsto \widecheck{f}$ is bijective and contravariant on $p$-morphisms. This bijection restricts to a bijection between $p$-isogenies and if $\mathcal{R}' = \mathcal{R}$ then it restricts to a bijection between $p$-Steinberg, resp., $p$-Frobenius, endomorphisms.
\end{rem}

\subsection{Reductive groups}
\begin{pa}\label{pa:isogeny}
Recall that a homomorphism $\phi : \bG \to \bG'$ between connected reductive $\mathbb{K}$-groups is called an \emph{isotypy} if $\Ker(\phi)$ is contained in the centre $Z(\bG)$ of $\bG$ and $\Image(\phi)$ contains the derived subgroup $\bG_{\der}'$ of $\bG'$. Moreover, we say $\phi$ is an \emph{isogeny} if it is surjective and has finite kernel. In particular, any isotypy restricts to an isogeny $\phi : \bG_{\der} \to \bG_{\der}'$ between derived subgroups. In what follows it will be convenient for us to consider pairs $(\bG,\bT)$ consisting of a connected reductive $\mathbb{K}$-group $\bG$ and a maximal torus $\bT \leqslant \bG$. We will define an isotypy $\phi : (\bG,\bT) \to (\bG',\bT')$ between such pairs to be an isotypy $\phi : \bG \to \bG'$ such that $\phi(\bT) \leqslant \bT'$.
\end{pa}

\begin{pa}\label{pa:functor-def}
By \cite[8.1.1]{springer:2009:linear-algebraic-groups} there exists, for any root $\alpha \in \Phi$, an isomorphism $x_{\alpha} : \mathbb{G}_a \to \bX_{\alpha}$ onto a unique closed subgroup $\bX_{\alpha} \leqslant \bG$ such that $tx_{\alpha}(k)t^{-1} = x_{\alpha}(\alpha(t)k)$ for any $k \in \mathbb{G}_a$ and $t \in \bT$. A family $(x_{\alpha})_{\alpha \in \Phi}$ of such isomorphisms will be called a \emph{realisation} of $(\bG,\bT)$; note that this is weaker then the corresponding notion defined in \cite[8.1.5]{springer:2009:linear-algebraic-groups}.
\end{pa}

\begin{pa}\label{pa:isog-to-isog-rd}
Now assume $\phi : (\bG,\bT) \to (\bG',\bT')$ is an isotypy and $(x_{\alpha})_{\alpha \in \Phi}$, resp., $(x_{\alpha}')_{\alpha \in \Phi}$, is a realisation of $(\bG,\bT)$, resp., $(\bG',\bT')$. We then have an induced $\mathbb{Z}$-module homomorphism $\phi^* : X(\bT') \to X(\bT)$ defined by $\phi^*(x) = x\circ\phi$. Arguing as in \cite[2.5]{steinberg:1999:the-isomorphism-and-isogeny-theorems} we see that there exists a $p$-morphism $\mathcal{R}(\phi) := (\phi^*,q,\tau) : \mathcal{R}(\bG',\bT') \to \mathcal{R}(\bG,\bT)$ and constants $c_{\alpha} \in \mathbb{G}_m$ such that
\begin{equation}\label{eq:action-of-phi}
\phi(x_{\alpha}(k)) = x_{\tau(\alpha)}'(c_{\alpha}k^{q(\alpha)}),
\end{equation}
see also \cite[Example A.4.4]{conrad-gabber-prasad:2010:pseudo-reductive-groups}. With this we have the following classical result which is a culimination of a generalisation of the isogeny theorem together with the existence theorem, see \cite[\S5]{steinberg:1999:the-isomorphism-and-isogeny-theorems}, \cite[II, 1.14]{jantzen:2003:representations-of-algebraic-groups}, and \cite[10.1.1]{springer:2009:linear-algebraic-groups}.
\end{pa}

\begin{thm}\label{thm:full-functor}
For any root datum $\mathcal{R}$ there exists a connected reductive $\mathbb{K}$-group $\bG$ and a maximal torus $\bT \leqslant \bG$ such that $\mathcal{R}(\bG,\bT) = \mathcal{R}$. Moreover, for any $p$-morphism $(f,q,\tau) : \mathcal{R}(\bG',\bT') \to \mathcal{R}(\bG,\bT)$ there exists an isotypy $\phi : (\bG,\bT) \to (\bG',\bT')$ such that $\mathcal{R}(\phi) = f$. If $\phi' : (\bG,\bT) \to (\bG',\bT')$ is another isotypy satisfying $\mathcal{R}(\phi') = (f,q,\tau)$ then there exists an element $t \in \bT$ such that $\phi' = \phi\circ\Inn t$, where $\Inn g : \bG \to \bG$ denotes the automorphism defined by $\Inn g(x) = gxg^{-1}$ for any $g,x \in \bG$.
\end{thm}

\begin{pa}
We will need the following characterisations of Steinberg and Frobenius endomorphisms, which are often used in the literature. For proofs of these statements we refer the reader to \cite[1.4.16, 1.4.17]{geck-malle:2016:reductive-groups-and-steinberg-maps}; see also \cite[3.17]{digne-michel:1991:representations-of-finite-groups-of-lie-type}.
\end{pa}

\begin{prop}\label{prop:characterisation-of-Steinberg-Frobenius}
Assume $\mathbb{K} = \overline{\mathbb{F}_p}$ then an isogeny $F : (\bG,\bT) \to (\bG,\bT)$ is a Steinberg/Frobenius endomorphism if and only if $\mathcal{R}(F) : \mathcal{R}(\bG,\bT) \to \mathcal{R}(\bG,\bT)$ is a $p$-Steinberg/$p$-Frobenius endomorphism.
\end{prop}

\begin{pa}
In the following sections we will need to remove the ambiguity over isotypies in \cref{thm:full-functor}. To do this we will need to add more information to the pair $(\bG,\bT)$. For this, let us choose a Borel subgroup $\bB \leqslant \bG$ containing $\bT$. The choice of Borel subgroup $\bB$ determines a natural set of simple roots $\Delta \subseteq \Phi \subseteq X(\bT)$. If $(x_{\alpha})_{\alpha \in \Delta}$ denotes a family of isomorphisms, as in \cref{pa:functor-def}, then we call the quadruple $(\bG,\bB,\bT,(x_{\alpha})_{\alpha \in \Delta})$ a \emph{pinned connected reductive group} or a \emph{pinning} of $(\bG,\bT)$. We will usually write $(x_{\alpha})$ instead of $(x_{\alpha})_{\alpha \in \Delta}$ for simplicity. An isotypy
\begin{equation*}
\phi : (\bG,\bB,\bT,(x_{\alpha})) \to (\bG',\bB',\bT',(x_{\alpha}'))
\end{equation*}
of pinned groups is then defined to be an isotypy $\phi : \bG \to \bG'$ such that $\phi(\bB) \leqslant \bB'$, $\phi(\bT) \leqslant \bT'$, and $c_{\alpha}=1$ for all $\alpha \in \Delta$, where $c_{\alpha}$ is as in \cref{eq:action-of-phi}.
\end{pa}

\begin{pa}
On the side of root data we have the corresponding notion of a \emph{based root datum}. This is a sextuple $(X,\Phi,\Delta,\widecheck{X},\widecheck{\Phi},\widecheck{\Delta})$ where $\mathcal{R} = (X,\Phi,\widecheck{X},\widecheck{\Phi})$ is a root datum, $\Delta \subseteq \Phi$ is a set of simple roots, and $\widecheck{\Delta} \subseteq \widecheck{\Phi}$ is the image of $\Delta$ under the bijection $\widecheck{\phantom{x}} : \Phi \to \widecheck{\Phi}$. Note that $\widecheck{\Delta}$ is then also a set of simple coroots. We shall also denote the based root datum by $(\mathcal{R},\Delta,\widecheck{\Delta})$. A \emph{$p$-morphism} between based root data is then defined to be a $p$-morphism $(f,q,\tau) : \mathcal{R}' \to \mathcal{R}$ such that $\tau(\Delta) = \Delta'$. To any quadruple $(\bG,\bB,\bT,(x_{\alpha}))$ we have a corresponding based root datum $\widehat{\mathcal{R}}(\bG,\bB,\bT,(x_{\alpha})) = (\mathcal{R}(\bG,\bT),\Delta,\widecheck{\Delta})$. Moreover, for any morphism $\phi : (\bG,\bB,\bT,(x_{\alpha})) \to (\bG',\bB',\bT',(x_{\alpha}'))$ we obtain a corresponding $p$-morphism of based root data $\widehat{\mathcal{R}}(\phi) : \widehat{\mathcal{R}}(\bG',\bB',\bT',(x_{\alpha}')) \to \widehat{\mathcal{R}}(\bG,\bB,\bT,(x_{\alpha}))$. The following is a strengthened form of \cref{thm:full-functor}. Its proof is an easy consequence of \cref{thm:full-functor} and \cite[16.2, C]{humphreys:1975:linear-algebraic-groups} which we leave to the reader.
\end{pa}

\begin{prop}\label{prop:unique-morphism}
For any $p$-morphism $(f,q,\tau) : \widehat{\mathcal{R}}(\bG',\bB',\bT',(x_{\alpha}')) \to \widehat{\mathcal{R}}(\bG,\bB,\bT,(x_{\alpha}))$ there exists a unique isotypy $\phi : (\bG,\bB,\bT,(x_{\alpha})) \to (\bG',\bB',\bT',(x_{\alpha}'))$ such that $\widehat{\mathcal{R}}(\phi) = (f,q,\tau)$.
\end{prop}

\begin{rem}\label{rem:compat-real}
Assume $F : \bG \to \bG$ is a Steinberg endomorphism then there exists an $F$-stable maximal torus and Borel subgroup $\bT\leqslant \bB \leqslant \bG$. There then exists a pinning $(\bG,\bB,\bT,(x_{\alpha}))$ of $(\bG,\bT)$ such that $F : (\bG,\bB,\bT,(x_{\alpha})) \to (\bG,\bB,\bT,(x_{\alpha}))$ is an isotypy of pinned groups. We will say such a pinning is compatible with $F$.
\end{rem}

\subsection{Constructing new root data from old}
\begin{pa}\label{pa:root-data-constructions}
We need to recall some constructions which allow us to create new root data from existing root data; we assume $\mathcal{R}$ and $\mathcal{R}'$ are root data as in \cref{pa:p-morphism}. One can easily construct a torus from the root datum $\mathcal{R}$ by setting $\mathcal{R}^{\circ} = (X,\emptyset,\widecheck{X},\emptyset)$. We may also form the \emph{direct sum} $\mathcal{R}\oplus\mathcal{R}'$ which is defined to be the quadruple $(X\oplus X',\Phi\cup\Phi',\widecheck{X}\oplus\widecheck{X}',\widecheck{\Phi}\cup\widecheck{\Phi}')$, where we identify $\Phi$, resp., $\Phi'$, with its image under the canonical inclusion map $X \to X\oplus X'$, resp., $X' \to X\oplus X'$; similar identifications are made for the coroots. Note that for any $(x,x') \in X\oplus X'$ and $(y,y') \in \widecheck{X}\oplus\widecheck{X}'$ we have
\begin{equation*}
\langle (x,x'), (y,y') \rangle_{\mathcal{R}\oplus\mathcal{R}'} = \langle x,y\rangle_{\mathcal{R}} + \langle x', y' \rangle_{\mathcal{R}'}.
\end{equation*}
\end{pa}

\begin{pa}\label{pa:induced-root-datum}
Now let $A \subseteq X$ be a submodule such that $\Phi \subseteq A$ and denote by $\widecheck{A} = \Hom(A,\mathbb{Z})$ the dual of $A$. We will denote by $\iota_A : A \to X$ the natural inclusion and by $\widecheck{\iota}_A : \widecheck{X} \to \widecheck{A}$ the map defined by $\widecheck{\iota}_A(y) = \langle -,y\rangle_{\mathcal{R}}\circ \iota_A$. The following provides a way to construct a root datum from $A$ with the same underlying roots.
\end{pa}

\begin{lem}[{}{\cite[\S6.5]{sga:2011:sga3-TomeIII}}]\label{lem:induced-root-datum}
The quadruple $\mathcal{R}_A = (A,\Phi,\widecheck{A},\widecheck{\iota}_A(\widecheck{\Phi}))$ is a root datum, with respect to the canonical perfect pairing $\langle -,-\rangle_{\mathcal{R}_A} : A \times \widecheck{A} \to \mathbb{Z}$ defined by $\langle a,b\rangle_{\mathcal{R}_A} = b(a)$, called the root datum \emph{induced by $A$}. Moreover $\iota_A : A \to X$ is a homomorphism of root data.
\end{lem}

\begin{rem}
Let $B \subseteq \widecheck{X}$ be a submodule containing $\widecheck{\Phi}$ and denote by $\widecheck{B}$ the dual module $\Hom(B,\mathbb{Z})$. If $\iota_B : B \to \widecheck{X}$ is the natural inclusion map and $\widecheck{\iota}_B : X \to \widecheck{B}$ is the map defined by $\widecheck{\iota}_B(x) = \langle x,-\rangle_{\mathcal{R}}\circ\iota_B$ then the quadruple $\mathcal{R}^B = (\widecheck{B},\widecheck{\iota}_B(\Phi),B,\widecheck{\Phi})$ is a root datum called the root datum \emph{co-induced by $B$}. Note that $\mathcal{R}^B$ is just the dual of the root datum $\widecheck{\mathcal{R}}_B$ induced by $B$.
\end{rem}

\begin{pa}\label{pa:radical}
Now assume $A \subseteq X$ is any \emph{subset} then we define submodules
\begin{align*}
A^{\top} &= \{x \in X \mid nx \in \mathbb{Z}A\text{ for some integer }n>0\} \subseteq X,\\
A^{\perp} &= \{y \in \widecheck{X} \mid \langle x,y\rangle_{\mathcal{R}} = 0\text{ for all }x \in A\} \subseteq \widecheck{X}.
\end{align*}
Note that $A^{\top}/\mathbb{Z}A = \Tor(X/\mathbb{Z}A)$. If $B \subseteq \widecheck{X}$ is a subset then the submodules $B^{\top} \subseteq \widecheck{X}$ and $B^{\perp} \subseteq X$ are defined in exactly the same way. Now assume $\Phi \subseteq A \subseteq X$ is a submodule as in \cref{pa:induced-root-datum} so that $A = A^{\top}$, in particular the quotient $X/A$ is a free module. In this case $\widecheck{\iota}_A$ is surjective and we have $\mathcal{R}_A$ is isomorphic to the root datum $(A,\Phi,\widecheck{X}/A^{\perp},\widecheck{\Phi})$ with the pairing being that induced by $\langle -,-\rangle_{\mathcal{R}}$. Here we implicitly identify $\widecheck{\Phi}$ with its image under the natural map $\widecheck{X} \to \widecheck{X}/A^{\perp}$. Following \cite[8.1.8, 8.1.9]{springer:2009:linear-algebraic-groups} we define the \emph{radical} of $\mathcal{R}$ to be the torus $\mathcal{R}_{\rad} = (\mathcal{R}^{\widecheck{\Phi}^{\perp}})^{\circ} = (X/\Phi^{\top},\emptyset,\widecheck{\Phi}^{\perp},\emptyset)$; note that $(\widecheck{\Phi}^{\top})^{\perp} = \Phi^{\top}$. Moreover, we define the \emph{derived datum} to be the semisimple datum $\mathcal{R}_{\mathrm{der}} = \mathcal{R}^{\widecheck{\Phi}^{\top}} = (X/\widecheck{\Phi}^{\perp},\Phi,\widecheck{\Phi}^{\top},\widecheck{\Phi})$.
\end{pa}

\section{Central Products of Root Data}\label{sec:central-products}
\begin{pa}\label{pa:fibre-prod}
In this section we introduce a construction of root data which is inherited from the fibre product of $\mathbb{Z}$-modules. For this let $\mathcal{R}_i = (X_i,\Phi_i,\widecheck{X}_i,\widecheck{\Phi}_i)$ be a root datum with $i \in \{1,2\}$ and let $A$ be a $\mathbb{Z}$-module equipped with two \emph{surjective} homomorphisms $h_i : X_i \to A$ such that $\Phi_i \subseteq \Ker(h_i)$. With respect to the triple $(A,h_1,h_2)$ we may construct the fibre product $X_1 \oplus_{(A,h_1,h_2)} X_2$ in the category of $\mathbb{Z}$-modules. In other words, we have a commutative diagram
\begin{equation}\label{eq:fibre-prod}
\begin{tikzpicture}[baseline=(current  bounding  box.center)]
\matrix (m) [matrix of math nodes, row sep=3em, column sep=3em, text height=1.5ex, text depth=0.25ex]
{ X_1 \oplus_{(A,h_1,h_2)} X_2 & X_1\\
  X_2 & A\\ };

\path [>=stealth,->]	(m-1-1) edge node[font=\small,above,text height=1.5ex,text depth=0.25ex,->] {$p_1$} (m-1-2)
		(m-1-1) edge node[font=\small,left,text height=1.5ex,text depth=0.25ex] {$p_2$} (m-2-1)
		(m-1-2) edge node[font=\small,right,text height=1.5ex,text depth=0.25ex] {$h_1$} (m-2-2)
		(m-2-1) edge node[font=\small,below,text height=1.5ex,text depth=0.25ex] {$h_2$} (m-2-2);
\end{tikzpicture}
\end{equation}
where $p_i : X_1 \oplus_{(A,h_1,h_2)} X_2 \to X_i$ is the projection map. Concretely we have
\begin{equation*}
X_1 \oplus_{(A,h_1,h_2)} X_2 = \{(x_1,x_2) \in X_1 \oplus X_2 \mid h_1(x_1) = h_2(x_2)\}.
\end{equation*}
If the maps $h_i$ are clear from context then we will usually write $X_1 \oplus_A X_2$ instead of $X_1 \oplus_{(A,h_1,h_2)} X_2$.
\end{pa}

\begin{rem}\label{rem:surj-map}
As we assume that $h_i : X_i \to A$ is surjective it is clear that the projection maps $p_i : X_1 \oplus_{(A,h_1,h_2)} X_2 \to X_i$ are also surjective. In fact, they are simply the restriction of the usual projection maps $p_i : X_1\oplus X_2 \to X_i$. We also note that $X_1 \oplus_{(A,h_1,h_2)} X_2 = X_1 \oplus X_2$ if and only if $A = \{0\}$.
\end{rem}

\begin{pa}\label{pa:fibre-prod-2}
Now consider the direct sum $\mathcal{R}_1\oplus \mathcal{R}_2 = (X_1\oplus X_2,\Phi,\widecheck{X}_1\oplus \widecheck{X}_2,\widecheck{\Phi})$ of root data, as in \cref{pa:root-data-constructions}. As $\Ker(h_i)$ contains the root lattice $\mathbb{Z}\Phi_i$ we clearly have $X_1 \oplus_{(A,h_1,h_2)} X_2$ contains the roots $\Phi = \{(\alpha_1,0),(0,\alpha_2) \mid \alpha_i \in \Phi_i\}$ so if $B := X_1 \oplus_{(A,h_1,h_2)} X_2$ then we may form the root datum
\begin{equation*}
\mathcal{R}_1 \oplus_{(A,h_1,h_2)} \mathcal{R}_2 := (\mathcal{R}_1\oplus\mathcal{R}_2)_B = (B,\Phi,\widecheck{B},\widecheck{\iota}_B(\widecheck{\Phi}))
\end{equation*}
induced by $B$. We say that $\mathcal{R}_1 \oplus_{(A,h_1,h_2)} \mathcal{R}_2$ is the \emph{central product} of $\mathcal{R}_1$ and $\mathcal{R}_2$ over $(A,h_1,h_2)$. Again we will usually denote $\mathcal{R}_1 \oplus_{(A,h_1,h_2)} \mathcal{R}_2$ by $\mathcal{R}_1 \oplus_A \mathcal{R}_2$ if the maps $h_1$ and $h_2$ are clear from context. We end this section by recording the following simple lemma.
\end{pa}

\begin{lem}\label{lem:fibre-prod}
Assume the notation of \cref{pa:fibre-prod,pa:fibre-prod-2} then the projection map $p_i : \mathcal{R}_1 \oplus_A \mathcal{R}_2 \to \mathcal{R}_i$ is a surjective homomorphism of root data. Furthermore, we have
\begin{align*}
\Ker(p_1) &= \{(0,x) \in X_1 \oplus_A X_2 \mid x \in \Ker(h_2)\},\\
\Ker(p_2) &= \{(x,0) \in X_1 \oplus_A X_2 \mid x \in \Ker(h_1)\},
\end{align*}
and $(X_1 \oplus_A X_2)/\Ker(p_i) \cong X_i$ has no torsion.
\end{lem}

\begin{proof}
The usual projection map $p_i' : X_1 \oplus X_2 \to X_i$ defines a homomorphism of root data because $\widecheck{p}_i' : \widecheck{X}_i \to \widecheck{X}_1 \oplus \widecheck{X}_2$ is simply the canonical inclusion map. Therefore as $p_i = p_i'\circ\iota_B$, where $\iota_B : B \to X_1 \oplus X_2$ is the inclusion, we have $p_i$ is a homomorphism of root data by \cref{lem:induced-root-datum}.
\end{proof}

\section{Structure and Classification of Root Data}\label{sec:struct-class-root-data}
\begin{pa}\label{pa:def-param-I}
Let is denote by $\mathscr{I}$ the class of all triples $(\mathcal{R},\mathcal{T},K)$ such that:
\begin{itemize}
	\item $\mathcal{R} = (X,\Phi,\widecheck{X},\widecheck{\Phi})$ is a semisimple root datum,
	\item $K \subseteq X$ is a submodule containing $\Phi$,
	\item and $\mathcal{T} = (T,\emptyset,\widecheck{T},\emptyset)$ is a torus such that there exists a surjective $\mathbb{Z}$-module homomorphism $T \to X/K$.
\end{itemize}
We define an equivalence class on $\mathscr{I}$ by setting $(\mathcal{R}_1,\mathcal{T}_1,K_1) \sim (\mathcal{R}_2,\mathcal{T}_2,K_2)$ if there exist isomorphisms $\varphi : \mathcal{R}_1 \to \mathcal{R}_2$ and $\psi : \mathcal{T}_1 \to \mathcal{T}_2$ such that $\varphi(K_1) = K_2$. The resulting set of equivalence classes is denoted by $\mathscr{I}/{\sim}$ and we denote by $[\mathcal{R},\mathcal{T},K]$ the equivalence class containing $(\mathcal{R},\mathcal{T},K) \in \mathscr{I}$.
\end{pa}

\begin{pa}
Before moving on let us note that the equivalence classes $\mathscr{I}/{\sim}$ can be described in more concrete terms. Consider pairs $(\Phi,V)$ consisting of a real Euclidean vector space $V$, with inner product $(-|-) : V \times V \to \mathbb{R}$, and a crystallographic root system $\Phi \subseteq V$. Associated to $\Phi$ we have its weight lattice $\Omega \subseteq V$ and fundamental group $\Omega/\mathbb{Z}\Phi$. For each subgroup $X/\mathbb{Z}\Phi \leqslant \Omega/\mathbb{Z}\Phi$ we have a corresponding semisimple root datum $\mathcal{R}_X = (X,\Phi,\widecheck{X},\widecheck{\Phi})$. Here we have $\widecheck{X} = \Hom(X,\mathbb{Z})$ and
\begin{equation*}
\widecheck{\Phi} = \left\{\left. 2\frac{(-|\alpha)}{(\alpha|\alpha)} \,\right|\, \alpha \in \Phi\right\}.
\end{equation*}
\end{pa}

\begin{pa}\label{pa:concrete-descr}
The automorphism group $\Aut(\Phi) \leqslant \GL(V)$ is the stabiliser of $\Phi$. This group stabilises the weight lattice $\Omega$ hence it acts on the fundamental group $\Omega/\mathbb{Z}\Phi$. For the fixed pair $(\Phi,V)$ consider the set $\mathscr{J}_{(\Phi,V)}$ of triples $(X/\mathbb{Z}\Phi,K/\mathbb{Z}\Phi,n)$ where $X/\mathbb{Z}\Phi$ and $K/\mathbb{Z}\Phi$ are subgroups of the fundamental group $\Omega/\mathbb{Z}\Phi$ and $n \geqslant 0$ is an integer such that $n$ is greater than or equal to the number of invariant factors of $X/K$. The group $\Aut(\Phi)$ acts on $\mathscr{J}_{(\Phi,V)}$ by acting simultaneously on the first two factors. The map $\mathscr{J}_{(\Phi,V)} \to \mathscr{I}$ defined by $(X/\mathbb{Z}\Phi,K/\mathbb{Z}\Phi,n) \mapsto (\mathcal{R}_X,\mathcal{T}_n,K)$, where $\mathcal{T}_n = (\mathbb{Z}^n,\emptyset,\mathbb{Z}^n,\emptyset)$ is a torus of rank $n$, defines a bijection
\begin{equation*}
\bigsqcup_{(\Phi,V)/{\cong}} \left(\mathscr{J}_{(\Phi,V)}/\Aut(\Phi) \right) \to \mathscr{I}/{\sim}.
\end{equation*}
Here we have $(\Phi_1,V_1) \cong (\Phi_2,V_2)$ if there exists an isometry $V_1 \to V_2$ mapping $\Phi_1$ onto $\Phi_2$.
\end{pa}

\begin{rem}
The Weyl group $W(\Phi)$ of $\Phi$ is a normal subgroup of $\Aut(\Phi)$ which acts trivially on $\Omega/\mathbb{Z}\Phi$ so the action of $\Aut(\Phi)$ on $\Omega/\mathbb{Z}\Phi$ factors through the quotient $\Aut(\Phi)/W(\Phi)$. Moreover, this quotient is isomorphic to the automorphism group of the underlying Dynkin diagram of $\Phi$, see \cite[VI, \S4, 2, Cor.\ to Prop.\ 1]{bourbaki:2002:lie-groups-chap-4-6}.
\end{rem}

\begin{pa}\label{pa:iso-R-subset}
If $\mathcal{R} = (X,\Phi,\widecheck{X},\widecheck{\Phi})$ is a root datum then we get a triple $(\mathcal{R}_{\der},\mathcal{R}_{\rad},(\Phi^{\top}\oplus\widecheck{\Phi}^{\perp})/\Phi^{\perp}) \in \mathscr{I}$, c.f., \cref{pa:def-param-I,pa:radical}. This triple is contained in $\mathscr{I}$ because $\mathcal{R}_{\rad} = (X/\Phi^{\top},\emptyset,\widecheck{\Phi}^{\perp},\emptyset)$ and we have a surjective homomorphism
\begin{equation*}
X/\Phi^{\top} \to X/(\Phi^{\top} \oplus \widecheck{\Phi}^{\perp}) \cong (X/\widecheck{\Phi}^{\perp})/((\Phi^{\top} \oplus \widecheck{\Phi}^{\perp})/\widecheck{\Phi}^{\perp}).
\end{equation*}
If $f : \mathcal{R}_1 \to \mathcal{R}_2$ is an isomorphism between root data $\mathcal{R}_i = (X_i,\Phi_i,\widecheck{X}_i,\widecheck{\Phi}_i)$ then this naturally induces isomorphisms $\varphi : (\mathcal{R}_1)_{\der} \to (\mathcal{R}_1)_{\der}$ and $\psi : (\mathcal{R}_1)_{\rad} \to (\mathcal{R}_1)_{\rad}$. Moreover, we have
\begin{equation*}
\varphi((\Phi_1^{\top}\oplus\widecheck{\Phi}_1^{\perp})/\widecheck{\Phi}_1^{\perp}) = (\Phi_2^{\top}\oplus\widecheck{\Phi}_2^{\perp})/\widecheck{\Phi}_2^{\perp}.
\end{equation*}
Thus, if we fix an equivalence class $[\mathcal{R},\mathcal{T},K] \in \mathscr{I}/{\sim}$ then we have a well-defined subset
\begin{equation*}
\mathscr{R}[\mathcal{R},\mathcal{T},K] = \{\mathcal{R}' = (X',\Phi',\widecheck{X}',\widecheck{\Phi}') \in \mathscr{R} \mid (\mathcal{R}'_{\der},\mathcal{R}'_{\rad},(\Phi'^{\top}\oplus\widecheck{\Phi}'^{\perp})/\Phi'^{\perp}) \sim (\mathcal{R},\mathcal{T},K)\}.
\end{equation*}
Thus we obtain a partition
\begin{equation}\label{eq:partition-root-data-iso-cls}
\mathscr{R} = \bigsqcup_{[\mathcal{R},\mathcal{T},K] \in \mathscr{I}/\sim}\mathscr{R}[\mathcal{R},\mathcal{T},K].
\end{equation}
\end{pa}

\begin{pa}\label{pa:def-O}
To solve \cref{prob:iso-root-data} it is clearly sufficient, given \cref{eq:partition-root-data-iso-cls}, to give a parameterisation of the elements in a given set $\mathscr{R}[\mathcal{R},\mathcal{T},K]$. For this we will need to understand the structure of an arbitrary root datum. We will do this using the central products considered in the previous section. We begin with the following definition which will appear again later.
\end{pa}

\begin{definition}\label{def:der-embed-rd}
If $\mathcal{R}_i = (X_i,\Phi_i,\widecheck{X}_i,\widecheck{\Phi}_i)$ are root data with $i \in \{1,2\}$ then a homomorphism of root data $f : \mathcal{R}_2 \to \mathcal{R}_1$ is said to be a \emph{derived embedding} if $f : X_2 \to X_1$ is surjective. If $\mathcal{R}_1$ is endowed with a $p$-Steinberg/$p$-Frobenius endomorphism $\phi_1 : \mathcal{R}_1 \to \mathcal{R}_1$ then we additionally require that there exists a $p$-Steinberg/$p$-Frobenius endomorphism $\phi_2 : \mathcal{R}_2 \to \mathcal{R}_2$ such that $f\circ\phi_1 = \phi_2\circ f$.
\end{definition}

\begin{lem}\label{lem:iso-of-der-embed}
Let $f : \mathcal{R}_2 \to \mathcal{R}_1$ be a derived embedding of root data. Let us set $A = X_1/f(\Phi_2^{\top})$ and let $h_2 : X_2/\Phi_2^{\top} \to A$ be the map defined by $h_2(x+\Phi_2^{\top}) = f(x)+f(\Phi_2^{\top})$. If $h_1 : X_1 \to A$ is the natural projection map then the homomorphism $\phi : X_2 \to X_1 \oplus (X_2/\Phi_2^{\top})$ defined by $\phi(x) = (f(x),x+\Phi_2^{\top})$ defines an isomorphism of root data
\begin{equation*}
\phi : \mathcal{R}_2 \to \mathcal{R}_1 \oplus_{(A,h_1,h_2)} (\mathcal{R}_2)_{\rad}.
\end{equation*}
Moreover, we have $f = p_1\circ\phi$ where $p_1 : \mathcal{R}_1 \oplus_{(A,h_1,h_2)} (\mathcal{R}_2)_{\rad} \to \mathcal{R}_1$ is the projection map.
\end{lem}

\begin{proof}
Let us denote $\mathcal{R}_1 \oplus_{(A,h_1,h_2)} (\mathcal{R}_2)_{\rad}$ by $(B,\Phi,\widecheck{B},\widecheck{\Phi})$. It's clear that $\phi(X_2) \subseteq B$. Now assume $(x_1,x_2+\Phi_2^{\top}) \in B$ then by the surjectivity of $f$ there exists an element $x_2' \in X_2$ such that $x_1 = f(x_2')$. By assumption there exists an element $a \in \Phi_2^{\top}$ such that $f(x_2') = f(x_2) + f(a) = f(x_2+a)$. With this we see that $\phi(x_2+a) = (x_1,x_2+\Phi_2^{\top})$ so $\phi(X_2) = B$.

Now assume $x \in \Ker(\phi)$ then certainly $x \in \Phi_2^{\top} \cap \Ker(f)$. However, this means there exists an integer $n>0$ such that $nx \in \mathbb{Z}\Phi_2$ and $f(nx) = nf(x) = 0$. As $f$ restricts to a bijection $\Phi_2 \to \Phi_1$ it restricts to an isomorphism $\mathbb{Z}\Phi_2 \to \mathbb{Z}\Phi_1$. This means $nx = 0$ so $x = 0$ because $X_2$ is a free module. Thus we have shown that the map is an isomorphism of $\mathbb{Z}$-modules. The fact that $\phi$ is an isomorphism of root data follows immediately from the fact that $f$ is a homomorphism of root data. The final statement is also clear.
\end{proof}

\begin{cor}[{}{see \cite[8.1.10]{springer:2009:linear-algebraic-groups}}]\label{cor:structure-root-datum}
Let $\mathcal{R} = (X,\Phi,\widecheck{X},\widecheck{\Phi})$ be a root datum and let $f : \mathcal{R} \to \mathcal{R}_{\der}$ be the derived embedding defined by the projection map $f : X \to X/\widecheck{\Phi}^{\perp}$. Then we have an isomorphism $\phi : \mathcal{R} \to \mathcal{R}_{\der} \oplus_{(A,h_1,h_2)} \mathcal{R}_{\rad}$, defined by $\phi(x) = (x+\widecheck{\Phi}^{\perp},x+\Phi^{\top})$, where $A = X/(\Phi^{\top}\oplus\widecheck{\Phi}^{\perp})$ and $h_1 : X/\widecheck{\Phi}^{\perp} \to A$ and $h_2 : X/\Phi^{\top} \to A$ are defined by $h_1(x+\widecheck{\Phi}^{\perp}) = x+(\Phi^{\top} \oplus \widecheck{\Phi}^{\perp})$ and $h_2(x+\Phi^{\top}) = x+(\Phi^{\top} \oplus \widecheck{\Phi}^{\perp})$.
\end{cor}

\begin{proof}
This follows immediately from \cref{lem:iso-of-der-embed} by noting that we have an isomorphism
\begin{equation*}
(X/\widecheck{\Phi}^{\perp})/f(\Phi^{\top}) \cong X/(\Phi^{\top}\oplus\widecheck{\Phi}^{\perp}).
\end{equation*}
\end{proof}

\begin{pa}
By \cref{cor:structure-root-datum} we have, up to isomorphism, every root datum is a central product $\mathcal{R} \oplus_{(A,h_1,h_2)} \mathcal{T}$ where $\mathcal{R} = (X,\Phi,\widecheck{X},\widecheck{\Phi})$ is semisimple and $\mathcal{T} = (T,\emptyset,\widecheck{T},\emptyset)$ is a torus. Note that as $A$ is isomorphic to a quotient of $X/\mathbb{Z}\Phi$ we must have $A$ is finite because $\mathcal{R}$ is semisimple. The following notes that we can recover $\mathcal{R}$, $\mathcal{T}$, and $A$, directly from $\mathcal{R} \oplus_{(A,h_1,h_2)} \mathcal{T}$.
\end{pa}

\begin{lem}\label{lem:recovery}
Assume $\mathcal{R}' := \mathcal{R} \oplus_{(A,h_1,h_2)} \mathcal{T} = (B,\Phi',\widecheck{B},\widecheck{\Phi}')$ where $\mathcal{R} = (X,\Phi,\widecheck{X},\widecheck{\Phi})$ is a semisimple root datum, $\mathcal{T} = (T,\emptyset,\widecheck{T},\emptyset)$ is a torus, and $A$ is a finite $\mathbb{Z}$-module. Then the following hold:
\begin{enumerate}
	\item $\Ker(p_1) = \widecheck{\Phi}'^{\perp}$ so $p_1$ factors through an isomorphism of root data $\mathcal{R}_{\der}' \to \mathcal{R}$,
	\item $\Ker(p_2) = \Phi'^{\top}$ so $p_2$ factors through an isomorphism of root data $\mathcal{R}_{\rad}' \to \mathcal{T}$,
	\item $\Ker(h_1\circ p_1) = \Ker(h_2\circ p_2) = \Phi'^{\top} \oplus \widecheck{\Phi}'^{\perp}$ so $h_1\circ p_1 = h_2\circ p_2$ factors through an isomorphism of abelian groups $B/(\Phi'^{\top} \oplus \widecheck{\Phi}'^{\perp}) \to A$.
%
%
\end{enumerate}
\end{lem}

\begin{proof}
(a). As $\mathcal{T}$ has no roots we see that $(x,t) \in \widecheck{\Phi}'^{\perp}$ if and only if $x \in \widecheck{\Phi}^{\perp} \subseteq X$ but $\widecheck{\Phi}^{\perp} = \{0\}$ because $\mathcal{R}$ is semisimple. Hence $\Ker(p_1) = \widecheck{\Phi}'^{\perp}$. We leave (b) to the reader.

(c). By the commutativity of the diagram in \cref{eq:fibre-prod}, \cref{lem:fibre-prod}, and parts (a) and (b) it follows that
\begin{equation*}
\Ker(h_1\circ p_1) = \Ker(h_2\circ p_2) = \Ker(p_1)\oplus\Ker(p_2) = \widecheck{\Phi}'^{\perp} \oplus \Phi'^{\top}.
\end{equation*}
As the composition $h_1\circ p_1 = h_2\circ p_2$ is surjective, see \cref{rem:surj-map}, the statement follows.
%
\end{proof}

\begin{cor}\label{cor:iso-central-prod}
Let $\mathcal{R}_i' := \mathcal{R}_i \oplus_{(A_i,h_i,f_i)} \mathcal{T}_i = (B_i,\Phi_i',\widecheck{B}_i,\widecheck{\Phi}_i')$ be a central product where $\mathcal{R}_i = (X_i,\Phi_i,\widecheck{X}_i,\widecheck{\Phi}_i)$ is a semisimple root datum and $\mathcal{T}_i = (T_i,\emptyset,\widecheck{T}_i,\emptyset)$ is a torus with $i \in \{1,2\}$. If $\zeta : \mathcal{R}_1' \to \mathcal{R}_2'$ is a $p$-morphism then there exist $p$-morphisms $\zeta_1 : \mathcal{R}_1 \to \mathcal{R}_2$ and $\zeta_2 : \mathcal{T}_1 \to \mathcal{T}_2$ and a homomorphism $\zeta_3 : A_1 \to A_2$ such that the following hold:
\begin{enumerate}
	\item $\zeta = (\zeta_1\oplus \zeta_2)|_{B_1}$,
	\item $\zeta_3\circ h_1 = h_2\circ \zeta_1$,
	\item $\zeta_3\circ f_1 = f_2\circ \zeta_2$.
\end{enumerate}
In particular, we have $\zeta_1(\Ker(h_1)) \subseteq \Ker(h_2)$. If $\zeta$ is an isomorphism then so is each $\zeta_i$ and we have $\zeta_1(\Ker(h_1)) = \Ker(h_2)$.
\end{cor}

\begin{proof}
As $\zeta$ is a $p$-morphism it is clear that $\zeta(\widecheck{\Phi}_1'^{\perp}) \subseteq \widecheck{\Phi}_2'^{\perp}$ and $\zeta(\Phi_1'^{\top}) \subseteq \Phi_2'^{\top}$ so we obtain induced $\mathbb{Z}$-module homomorphisms $\zeta : B_1/\widecheck{\Phi}_1'^{\perp} \to B_2/\widecheck{\Phi}_2'^{\perp}$, $\zeta : B_1/\Phi_1'^{\top} \to B_2/\Phi_2'^{\top}$, and $\zeta : B_1/(\widecheck{\Phi}_1'^{\perp} \oplus \Phi_1'^{\top}) \to B_2/(\widecheck{\Phi}_2'^{\perp} \oplus \Phi_2'^{\top})$. In particular, these define $p$-morphisms $(\mathcal{R}_1')_{\der} \to (\mathcal{R}_2')_{\der}$ and $(\mathcal{R}_1')_{\rad} \to (\mathcal{R}_2')_{\rad}$ and we get the following diagram where each square commutes
\begin{equation*}
\begin{tikzpicture}[baseline=(current  bounding  box.center)]
\matrix (m) [matrix of math nodes, row sep=3em, column sep=3em, text height=1.5ex, text depth=0.25ex]
{ X_1 & B_1/\widecheck{\Phi}_1'^{\perp} & B_2/\widecheck{\Phi}_2'^{\perp} & X_2\\
  A_1 & B_1/(\Phi_1'^{\top} \oplus \widecheck{\Phi}_1'^{\perp}) & B_2/(\Phi_2'^{\top} \oplus \widecheck{\Phi}_2'^{\perp}) & A_2\\
 T_1 & B_1/\Phi_1'^{\top} & B_2/\Phi_2'^{\top} & T_2\\ };

\path [>=stealth,->]	
        (m-1-1) edge node[font=\small,left,text height=1.5ex,text depth=0.25ex,->] {$h_1$} (m-2-1)
        (m-1-2) edge (m-1-1)
        (m-2-2) edge (m-2-1)
        (m-1-2) edge node[font=\small,above,text height=1.5ex,text depth=0.25ex,->] {$\zeta$} (m-1-3)
        (m-1-2) edge (m-2-2)
        (m-2-2) edge node[font=\small,above,text height=1.5ex,text depth=0.25ex,->] {$\zeta$} (m-2-3)
        (m-1-3) edge (m-1-4)
		(m-1-3) edge (m-2-3)
		(m-1-4) edge node[font=\small,left,text height=1.5ex,text depth=0.25ex] {$h_2$} (m-2-4)
		(m-2-3) edge (m-2-4)
		(m-3-2) edge (m-3-1)
		(m-3-2) edge node[font=\small,above,text height=1.5ex,text depth=0.25ex,->] {$\zeta$} (m-3-3)
		(m-3-3) edge (m-3-4)
		(m-3-1) edge node[font=\small,left,text height=1.5ex,text depth=0.25ex] {$f_1$} (m-2-1)
		(m-3-2) edge (m-2-2)
		(m-3-3) edge (m-2-3)
		(m-3-4) edge node[font=\small,left,text height=1.5ex,text depth=0.25ex] {$f_2$} (m-2-4);
\end{tikzpicture}
\end{equation*}
Here the unmarked horizontal maps are the isomorphisms given by \cref{lem:recovery}; the unmarked vertical maps are the canonical projections. The top row now defines $\zeta_1$, the middle row defines $\zeta_3$, and the bottom row defines $\zeta_2$.
\end{proof}

\begin{rem}\label{rem:lifting-Steinberg-Frob}
If $\mathcal{R}_1' = \mathcal{R}_2'$ in \cref{cor:iso-central-prod} and $\zeta$ is a $p$-Steinberg, resp., $p$-Frobenius, endomorphism then so are $\zeta_1$ and $\zeta_2$.
\end{rem}

\begin{definition}\label{def:tame-auto}
Assume $\mathcal{R} = (X,\Phi,\widecheck{X},\widecheck{\Phi})$ is any root datum and $f : X \to A$ is a fixed surjective homomorphism of $\mathbb{Z}$-modules such that $\Phi \subseteq K := \Ker(f)$. We denote by $\Aut(\mathcal{R})_K$ those automorphisms $\tilde{\psi} \in \Aut(\mathcal{R})$ of the root datum satisfying $\tilde{\psi}(K) = K$. Any such automorphism determines an automorphism of $X/K \cong A$ so we have a homomorphism $\Aut(\mathcal{R})_K \to \Aut(A)$ whose image we denote by $\Aut_{(\mathcal{R},f)}(A)$. The elements of $\Aut_{(\mathcal{R},f)}(A)$ are said to be \emph{tame} with respect to $(\mathcal{R},f)$. In other words, $\Aut_{(\mathcal{R},f)}(A)$ consists of those automorphisms $\psi \in \Aut(A)$ for which there exists an automorphism $\tilde{\psi} \in \Aut(\mathcal{R})$ satisfying $\psi\circ f = f\circ\tilde{\psi}$.
\end{definition}

\begin{pa}
If one takes $\mathcal{R}$ to be a torus in \cref{def:tame-auto} then $\Phi = \emptyset$ so one is simply asking for a $\mathbb{Z}$-module automorphism $\tilde{\psi} : X \to X$ that lifts $\psi$ with respect to $f$. We give an example illustrating this point, which will be useful later on.
\end{pa}

\begin{exmp}\label{exmp:cyclic-group}
Assume $n \geqslant 1$ is an integer and set $A := \mathbb{Z}/n\mathbb{Z}$. We assume $\mathcal{T} = (T,\emptyset,\widecheck{T},\emptyset)$ is a torus and $f : T \to A$ is a fixed surjective homomorphism. Let $(t_1,\dots,t_r)$ be a basis of $T$ adapted to $f$, i.e., we have $(nt_1,t_2,\dots,t_r)$ is a basis of $\Ker(f)$. With this choice of basis we will identify $\Aut(T)$ with $\GL_r(\mathbb{Z})$.

Let us first assume that $r = 1$. If $k \in \mathbb{Z}$ then $\psi_k : A \to A$, defined by $\psi_k(x) = kx$ for all $x \in A$, is a homomorphism. Moreover, we have $\Aut(A) = \{\psi_k \mid 1 \leqslant k < n$ and $\gcd(k,n) = 1\} \cong \mathbb{Z}/\varphi(n)\mathbb{Z}$ where $\varphi(n)$ is the evaluation at $n$ of Euler's totient function. It's clear that $\Aut_{(\mathcal{T},f)}(A) = \{\psi_1,\psi_{n-1}\}$ because $\Aut(T) \cong \GL_1(\mathbb{Z}) = \{(1),(-1)\}$.

Now consider the case where $r = 2$ and again let $\psi_k \in \Aut(A)$. As $\gcd(k,n) = 1$ we have by B\'ezout's identity that there exist integers $a,b \in \mathbb{Z}$ such that $ka+nb = 1$. The matrix
\begin{equation*}
\begin{bmatrix}
k & -n\\
b & a
\end{bmatrix} \in \SL_2(\mathbb{Z})
\end{equation*}
thus determines an automorphism of $T$ lifting $\psi$. Hence, in this case we have $\Aut_{(\mathcal{T},f)}(A) = \Aut(A)$. If $r > 2$ then $T = \langle t_1,t_2\rangle\oplus \langle t_3,\dots,t_r\rangle$ so it's clear we again have $\Aut_{(\mathcal{T},f)}(A) = \Aut(A)$.
\end{exmp}

\begin{thm}\label{thm:smooth-reg-embed-classification}
Assume $(\mathcal{R},\mathcal{T},K) \in \mathscr{I}$ is a triple with $\mathcal{R} = (X,\Phi,\widecheck{X},\widecheck{\Phi})$ and fix a pair of surjective $\mathbb{Z}$-module homomorphisms $h : X \to X/K$ and $f : T \to X/K$ such that $\Ker(h) = K$. If $A = X/K$ then the map $\Aut(A) \to \mathscr{R}$ defined by $\psi \mapsto \mathcal{R} \oplus_{(A,h,\psi\circ f)} \mathcal{T}$ induces a bijection
\begin{equation*}
\Aut_{(\mathcal{R},h)}(A) \backslash \Aut(A) / \Aut_{(\mathcal{T},f)}(A) \to \mathscr{R}[\mathcal{R,T},K].
\end{equation*}
Here $\Aut_{(\mathcal{R},h)}(A)$ are those automorphisms of $A$ that are induced by an automorphism of $\mathcal{R}$ and similarly for $\Aut_{(\mathcal{T},f)}(A)$, see \cref{def:tame-auto} for details.
\end{thm}

\begin{proof}
By \cref{cor:structure-root-datum,lem:recovery} any root datum $\mathcal{R}' \in \mathscr{R}[\mathcal{R},\mathcal{T},K]$ is isomorphic to a central product $\mathcal{R} \oplus_{(X/K,h',f')} \mathcal{T}$ with $h' : X \to X/K$ and $f' : T \to X/K$ surjective homomorphisms with $\Ker(h') = K$. Now, choose a basis $(x_1,\dots,x_n)$ of $X$ adapted to $h : X \to A$. By assumption $\Ker(h) = \Ker(h')$ so we have an automorphism $\psi_1 \in \Aut(A)$, defined by $\psi_1(h(x_i)) = h'(x_i)$ for $1 \leqslant i \leqslant s$, which satisfies $h' = \psi_1 \circ h$.

Again we choose a basis $(t_1,\dots,t_n)$, resp., $(t_1',\dots,t_n')$, of $T$ adapted to $f$, resp., $f'$. As $T/\Ker(f) \cong T/\Ker(f') \cong A$ we have automorphisms $\lambda : T \to T$, defined by $\lambda(t_i') = t_i$, and $\psi_2 : A \to A$, defined by $\psi_2(f(t_i)) = f'(t_i')$, which satisfy $\psi_2 \circ f \circ \lambda = f'$. If $\psi = \psi_1^{-1}\circ\psi_2$ then we have an isomorphism
\begin{equation*}
X \oplus_{(A,h',f')} T \to X \oplus_{(A,h,\psi\circ f)} T = X \oplus_{(A,\psi_1\circ h,\psi_2\circ f)} T
\end{equation*}
defined by $(x,t) \mapsto (x,\lambda(t))$. This is clearly an isomorphism of root data $\mathcal{R} \oplus_{(A,h',f')} \mathcal{T} \to \mathcal{R} \oplus_{(A,h,\psi\circ f)} \mathcal{T}$ which shows the map is surjective.

Now assume $\psi_1,\psi_2 \in \Aut(A)$ are such that $\psi_1 = \zeta \circ \psi_2\circ\psi$ for some $\psi \in \Aut_{(\mathcal{T},f)}(A)$ and $\zeta \in \Aut_{(\mathcal{R},h)}(A)$. As $\psi$ and $\zeta$ are tame there exist automorphisms $\tilde{\psi} \in \Aut(\mathcal{T})$ and $\tilde{\zeta} \in \Aut(\mathcal{R})$ such that $\zeta \circ h = h \circ \tilde{\zeta}$ and $\psi \circ f = f \circ \tilde{\psi}$. As we also have $\zeta^{-1} \circ h = h \circ \tilde{\zeta}^{-1}$ one easily checks that $(x,t) \mapsto (\tilde{\zeta}^{-1}(x),\tilde{\psi}(t))$ defines an isomorphism of root data $\mathcal{R} \oplus_{(A,h,\psi_1\circ f)} \mathcal{T} \to \mathcal{R} \oplus_{(A,h,\psi_2\circ f)} \mathcal{T}$ which shows the map is well defined.

Finally let $\mathcal{R} \oplus_{(A,h,\psi_1\circ f)} \mathcal{T} = (B_1,\Phi_1,\widecheck{B}_1,\widecheck{\Phi}_1)$ and $\mathcal{R} \oplus_{(A,h,\psi_2\circ f)} \mathcal{T} = (B_2,\Phi_2,\widecheck{B}_2,\widecheck{\Phi}_2)$ and let us assume that we have an isomorphism $\zeta : \mathcal{R} \oplus_{(A,h,\psi_1\circ f)} \mathcal{T} \to \mathcal{R} \oplus_{(A,h,\psi_2\circ f)} \mathcal{T}$. By \cref{cor:iso-central-prod} we have automorphisms $\zeta_1 : X \to X$, $\zeta_2 : T \to T$, and $\zeta_3 : A \to A$ such that $\zeta_3 \circ h = h \circ \zeta_1$ and $\zeta_3 \circ (\psi_1\circ f) = (\psi_2\circ f)\circ \zeta_2$. If $\psi = \psi_2^{-1} \circ \zeta_3\circ\psi_1$ then we have $\psi\circ f = f\circ\zeta_2$ so $\psi \in \Aut_{(\mathcal{T},f)}(A)$ and moreover $\psi_1 = \zeta_3^{-1}\circ\psi_2 \circ \psi$. Note that $\zeta_3^{-1} \in \Aut_{(\mathcal{R},h)}(A)$ and $\psi \in \Aut_{(\mathcal{T},f)}(A)$ so the map is injective.
\end{proof}

\begin{prop}\label{prop:aut-ab-lift}
Assume $\mathcal{T} = (T,\emptyset,\widecheck{T},\emptyset)$ is a torus and $f : T \to A$ is a surjective homomorphism onto a finite $\mathbb{Z}$-module. If $A$ has $s > 0$ invariant factors and $\rk(\mathcal{T}) \geqslant s+1$ then $\Aut_{(\mathcal{T},f)}(A) = \Aut(A)$.
\end{prop}

\begin{proof}
Choose a basis $(t_1,\dots,t_n)$ of $T$ adapted to $f$. For $1 \leqslant i \leqslant s$ we denote by $A_i \leqslant A$ the cyclic submodule $\langle f(t_i) \rangle \cong \mathbb{Z}/d_i\mathbb{Z}$ so that $A = A_1\oplus\cdots \oplus A_s$. Now, assume $\psi \in \Aut(A)$ then $\{\psi(f(t_1)),\dots,\psi(f(t_s))\}$ is another generating set of $A$. By \cite[(2.2)]{diaconis-graham:1999:the-graph-of-generating-sets} there exist automorphisms $\tau \in \Aut(T)$ and $\gamma \in \Aut(A)$ such that $\psi \circ f = \gamma \circ f \circ \tau$ and $\gamma = \ID_{A_1}\oplus\cdots\oplus \ID_{A_{s-1}}\oplus \psi_k$. Here we use the notation of \cref{exmp:cyclic-group} so that $1 \leqslant k < d_s$ is an integer with $\gcd(k,d_s) = 1$. The construction of \cref{exmp:cyclic-group} clearly shows that $\gamma$ is tame with respect to $(\mathcal{T},f)$. Hence, $\psi \in \Aut_{(\mathcal{T},f)}(A)$ as desired.
\end{proof}

\begin{rem}\label{rem:auto-lift}
The proof of \cref{prop:aut-ab-lift}, together with \cref{exmp:cyclic-group}, shows that if $\Aut(\mathbb{Z}/d_s\mathbb{Z}) = \{\psi_1,\psi_{d_s-1}\}$ then we have $\Aut_{(\mathcal{T},f)}(A) = \Aut(A)$ even when $\rk(\mathcal{T}) = s$. Note that this condition on $\Aut(\mathbb{Z}/d_s\mathbb{Z})$ holds if and only if $d_s \in \{2,3,4,6\}$.
\end{rem}

\section{Smooth Regular Embeddings}\label{sec:smooth-reg-embed}
\begin{pa}
In \cite[\S1.21]{deligne-lusztig:1976:representations-of-reductive-groups} Deligne--Lusztig gave a construction which showed that regular embeddings always exist. We begin this section by showing that smooth regular embeddings always exist. Our approach is to mimic the construction of Deligne--Lusztig at the level of root data. We note that in the absence of Frobenius/Steinberg endomorphisms this exact construction has been given before by Martin in \cite[Theorem 4.5]{martin:1999:etale-slices-for-representation-varieties}. We start by defining the analogues of the notions of regular embedding and smooth regular embedding at the level of root data.
\end{pa}

\begin{definition}
Let $\mathcal{R}$ and $\mathcal{R}'$ be root data as in \cref{pa:p-morphism} then a derived embedding $f : \mathcal{R}' \to \mathcal{R}$, c.f., \cref{def:der-embed-rd}, is called a \emph{$p$-regular embedding} if $X'/\mathbb{Z}\Phi'$ has no $p'$-torsion and a \emph{smooth regular embedding} if $X'/\mathbb{Z}\Phi'$ has no torsion.
\end{definition}

\begin{lem}\label{lem:lift-frob}
Assume $\bG$ is endowed with a Steinberg/Frobenius endomorphism $F : \bG \to \bG$ and let $\bT\leqslant \bG$ be an $F$-stable maximal torus. Let $\pi : (\bG,\bT) \to (\bG',\bT')$ be an isotypy such that there exists a $p$-Steinberg/$p$-Frobenius endomorphism $\phi' : \mathcal{R}(\bG',\bT') \to \mathcal{R}(\bG',\bT')$ satisfying $\mathcal{R}(\pi\circ F) = \mathcal{R}(F)\circ\mathcal{R}(\pi) = \mathcal{R}(\pi)\circ\phi'$. Then there exists a Steinberg/Frobenius endomorphism $F' : (\bG',\bT') \to (\bG',\bT')$ satisfying $\pi\circ F = F'\circ\pi$ and $\mathcal{R}(F') = \phi'$. Similarly, if $\mathcal{R}(F\circ\pi) = \mathcal{R}(\pi)\circ\mathcal{R}(F) = \phi'\circ\mathcal{R}(\pi)$ then there exists a Steinberg/Frobenius endomorphism $F' : (\bG',\bT') \to (\bG',\bT')$ satisfying $F\circ\pi = \pi\circ F'$ and $\mathcal{R}(F') = \phi'$.
\end{lem}

\begin{proof}
By \cref{thm:full-functor,prop:characterisation-of-Steinberg-Frobenius} there exists a Steinberg/Frobenius endomorphisms $F' : (\bG',\bT') \to (\bG',\bT')$ such that $\mathcal{R}(F') = \phi'$. As $\mathcal{R}(\pi\circ F) = \mathcal{R}(F'\circ\pi)$ we have again by \cref{thm:full-functor} that there exists an element $t \in \bT$ such that $\pi\circ F = F'\circ\pi \circ \Inn t$. This implies $\pi\circ F = (F'\circ\Inn \pi(t)) \circ \pi$ so simply replacing $F'$ by $F'\circ\Inn \pi(t)$ we get the statement. The second statement is proved identically.
\end{proof}

\begin{prop}[{}{see \cite[1.7.8]{geck-malle:2016:reductive-groups-and-steinberg-maps}}]\label{prop:der-embeds-root-data}\label{pa:der-embed-induces-der-embed}
Let $\pi : \bG \to \bG'$ be an isotypy and assume $\bG$ is not endowed with a Steinberg/Frobenius endomorphism then $\pi$ is a derived embedding if and only if $\mathcal{R}(\pi) : \mathcal{R}(\bG',\bT') \to \mathcal{R}(\bG,\bT)$ is a derived embedding for some (any) maximal tori $\bT \leqslant \bG$ and $\bT' \leqslant \bG'$.
\end{prop}

\begin{lem}[{}{see \cite[Theorem 4.5]{martin:1999:etale-slices-for-representation-varieties}}]\label{lem:smooth-reg-embed}
There exists a smooth regular embedding $\pi : \bG \to \bG'$.
\end{lem}

\begin{proof}
Let $\bT \leqslant \bG$ be a maximal torus and let $\mathcal{R} = \mathcal{R}(\bG,\bT) = (X,\Phi,\widecheck{X},\widecheck{\Phi})$ be the root datum of $(\bG,\bT)$. Set $A = X/\mathbb{Z}\Phi$ and let $f : X \to X/\mathbb{Z}\Phi$ be the canonical projection map then we may form the central product $\mathcal{R}' = \mathcal{R} \oplus_{(A,f,f)} \mathcal{R}^{\circ} = (X',\Phi',\widecheck{X}',\widecheck{\Phi}')$. Now let $p_1 : \mathcal{R}' \to \mathcal{R}$ be the projection onto the first factor then by \cref{lem:fibre-prod} this is a surjective homomorphism of root data and we have $\Tor(X'/\mathbb{Z}\Phi') = \{0\}$. By \cref{thm:full-functor} there exists a pair $(\bG',\bT')$ such that $\mathcal{R}(\bG',\bT') = \mathcal{R}'$ and an isotypy $\pi : (\bG,\bT) \to (\bG',\bT')$ such that $\mathcal{R}(\pi) = p_1$. It follows from \cref{prop:der-embeds-root-data} that $\pi$ is a smooth regular embedding in the absence of Steinberg endomorphisms.

Now let us assume that $\bG$ is equipped with a Steinberg/Frobenius endomorphism $F : \bG \to \bG$ then we may assume that $\bT$ is chosen to be $F$-stable. We denote by $\phi' : X' \to X'$ the homomorphism defined by $\phi'(x,y) = (F^*(x),F^*(y))$, which is well defined as $F^*(\mathbb{Z}\Phi) \subseteq \mathbb{Z}\Phi$. It is clear that $\phi'$ induces a $p$-Steinberg/$p$-Frobenius endomorphism of $\mathcal{R}(\bG',\bT')$ such that $\mathcal{R}(F)\circ \mathcal{R}(\pi) = \mathcal{R}(\pi)\circ\phi'$. It follows from \cref{lem:lift-frob} that there exists a Steinberg/Frobenius endomorphism $F' : (\bG',\bT') \to (\bG',\bT')$ such that $\mathcal{R}(F') = \phi'$ and $F\circ\pi = \pi\circ F'$ so $\pi$ is a smooth regular embedding.
\end{proof}

\begin{exmp}\label{exmp:Sp4}
Assume now that $\bG = \Sp_4(\mathbb{K})$ and $\mathcal{R}(\bG,\bT) = (X,\Phi,\widecheck{X},\widecheck{\Phi})$ is the root datum of $\bG$ with respect to some maximal torus $\bT \leqslant \bG$. If $\Delta = \{\alpha_1,\alpha_2\}$ is a set of simple roots, with $\alpha_2$ the long root, then $X = \mathbb{Z}\omega_1\oplus\mathbb{Z}\omega_2$ and $\widecheck{X} = \mathbb{Z}\widecheck{\alpha}_1\oplus\mathbb{Z}\widecheck{\alpha}_2$ with $2\omega_1 = 2\alpha_1+\alpha_2$ and $\omega_2 = \alpha_1+\alpha_2$. Now let $\mathcal{R}' = (X',\Phi',\widecheck{X}',\widecheck{\Phi}')$ be the root datum constructed in the proof of \cref{lem:smooth-reg-embed} then
\begin{equation*}
X' = \{(a\omega_1+b\omega_2,c\omega_1+d\omega_2) \in X\oplus X \mid a-c \in 2\mathbb{Z}\} = \mathbb{Z}e_1 \oplus \mathbb{Z}e_2 \oplus \mathbb{Z}e_3 \oplus \mathbb{Z}e_4,
\end{equation*}
where $e_1 = (\omega_1,\omega_1)$, $e_2 = (\omega_2-\omega_1,\omega_1)$, $e_3 = (0,2\omega_1)$, and $e_4 = (0,\omega_2)$. If $(\widecheck{e}_1,\widecheck{e}_2,\widecheck{e}_3,\widecheck{e}_4)$ is the dual basis then the corresponding simple roots and coroots of $\mathcal{R}'$ are given by $\alpha_1' = e_1 - e_2$, $\alpha_2' = 2e_2 - e_3$, and $\widecheck{\alpha}_1' = \widecheck{e}_1 - \widecheck{e}_2$, $\widecheck{\alpha}_2' = \widecheck{e}_2$ respectively. From the calculations in \cite[Chapter 1]{luebeck:1993:thesis} we see that $\bG' \cong \CSp_4(\mathbb{K}) \times \mathbb{G}_m$.
\end{exmp}

\begin{pa}\label{pa:optimal-sit}
Although the construction of a smooth regular embedding in \cref{lem:smooth-reg-embed} works in all cases it has the downside that it does not attempt to minimise the dimension of $Z(\bG')$. In fact the construction gives a group whose centre has dimension the rank of $\bG$, which begs the question: How small can $\dim(Z(\bG'))$ be? This is particularly relevant when studying characters of finite reductive groups. Now, assume $\bG$ is semisimple and $\pi : (\bG,\bT) \to (\bG',\bT')$ is a smooth regular embedding. Let $\mathcal{R}(\bG,\bT) = (X,\Phi,\widecheck{X},\widecheck{\Phi})$ and $\mathcal{R}(\bG',\bT') = (X',\Phi',\widecheck{X}',\widecheck{\Phi}')$ then as $\bG$ is semisimple we have $X/\mathbb{Z}\Phi$ is finite; we assume it has $s\geqslant 0$ invariant factors.
\end{pa}

\begin{pa}
Let $f = \mathcal{R}(\pi) : \mathcal{R}(\bG',\bT') \to \mathcal{R}(\bG,\bT)$ be the corresponding smooth regular embedding of root data. As $f : X' \to X$ is surjective and $\Phi' \subseteq \Ker(f)$ we must have the induced map $f : X'/\mathbb{Z}\Phi' \to X/\mathbb{Z}\Phi$ is surjective. The quotient $X'/\mathbb{Z}\Phi'$ is a free module and we must have $\rk(X'/\mathbb{Z}\Phi') \geqslant s$ by \cref{lem:adapted-basis}. As the rank of $X'/\mathbb{Z}\Phi' \cong X(Z(\bG'))$ is the same as the dimension of $Z(\bG')$ we see that $\dim(Z(\bG')) \geqslant s$. It is an easy exercise with central products to show that this lower bound is sharp in the absence of Steinberg or Frobenius endomorphisms. The following shows that this bound is sharp if $\bG$ is simple and endowed with a Frobenius endomorphism (see \cref{lem:no-smooth-embedding-suzuki} for the case of Steinberg endomorphisms).
\end{pa}

\begin{prop}\label{prop:optimal-embed}
Assume $\bG$ is simple and $F : \bG \to \bG$ is a Frobenius endomorphism. Let $\mathcal{R} = \mathcal{R}(\bG,\bT) = (X,\Phi,\widecheck{X},\widecheck{\Phi})$ be the root datum of $\bG$ with respect to an $F$-stable maximal torus $\bT \leqslant \bG$. If $A := X/\mathbb{Z}\Phi$ has $s \geqslant 0$ invariant factors then there exists a smooth regular embedding $\pi : \bG \to \bG'$ such that $\dim(Z(\bG')) = s$.
\end{prop}

\begin{proof}
If $\bG$ is adjoint there is nothing to show so we can assume $s > 0$. As $F$ is a Frobenius endomorphism we have $\mathcal{R}(F) = q\tau : \mathcal{R} \to \mathcal{R}$ with $q$ an integral $p$-power and $\tau : X \to X$ a finite order automorphism. Clearly $\tau$ induces an automorphism $\tau : A \to A$. We assume $\mathcal{T} = (T,\emptyset,\widecheck{T},\emptyset)$ is a torus with $\rk(\mathcal{T}) = s$ and $f : T \to A$ is a surjective homomorphism. We claim that $\tau \in \Aut_{(\mathcal{T},f)}(A)$ is tame with respect to $(\mathcal{T},f)$. By the classification of simple algebraic groups the invariant factors of $A$ are given by $(n)$ or $(2,2)$ where $n \geqslant 1$ is an integer. If the invariant factors are $(2,2)$ or $(n)$ with $n \leqslant 3$ then the statement follows from \cref{rem:auto-lift}. If $A$ has invariant factors $(n)$ with $n > 3$ then $\bG$ is of type $\A_{n-1}$ and we have $\tau \in \{\psi_1,\psi_{n-1}\}$, so the statement follows from \cref{exmp:cyclic-group}.

Let $\tilde{\tau} \in \Aut(T)$ be a lift of $\tau$ so that $\tau\circ f = f\circ\tilde{\tau}$. We now define an automorphism $\psi : X \oplus T \to X\oplus T$ by setting $\psi(x,y) = (\tau(x),\tilde{\tau}(y))$. As the order of $\tau$ must divide the order of $\tilde{\tau}$ it's clear that $\psi$ has finite order so $F_1^* = q\psi$ is a $p$-Frobenius endomorphism of $\mathcal{R} \oplus \mathcal{T}$. This clearly restricts to a $p$-Frobenius endomorphism of the central product $\mathcal{R} \oplus_{(A,h,f)} \mathcal{T}$, where $h : X \to A$ is the natural projection map, satisfying $p_1\circ F_1^* = F^*\circ p_1$. Thus there exists the desired smooth regular embedding $\pi : \bG \to \bG'$ by \cref{lem:fibre-prod,lem:lift-frob,prop:der-embeds-root-data}.
\end{proof}

\begin{pa}
Note that the conclusion of \cref{prop:optimal-embed} no longer holds if we replace the assumption that $\bG$ is simple by the assumption that $\bG$ is semisimple. Indeed, assume $\bG = \SL_3(\mathbb{K}) \times \SL_5(\mathbb{K})$ then $A := X/\mathbb{Z}\Phi \cong \mathbb{Z}/3\mathbb{Z} \oplus \mathbb{Z}/5\mathbb{Z}$ so $A$ has invariant factors $(15)$. Assume $F : \bG \to \bG$ is a Frobenius endomorphism restricting to a split Frobenius endomorphism on $\SL_3(\mathbb{K})$ and a twisted Frobenius endomorphism on $\SL_5(\mathbb{K})$. Then $F^* = q\tau : X \to X$ is such that the automorphism $\tau$ induces the automorphism $\psi_4$ on $A \cong \mathbb{Z}/15\mathbb{Z}$, with the notation as in \cref{exmp:cyclic-group}. If $\mathcal{T} = (T,\emptyset,\widecheck{T},\emptyset)$ is a $1$-dimensional torus with a surjective homomorphism $f : T \to A$ then by \cref{exmp:cyclic-group} we see that $\psi_4$ is not tame with respect to $(\mathcal{T},f)$. One easily sees that this implies that there is no smooth regular embedding $\pi : \bG \to \bG'$ with $\dim(Z(\bG')) = 1$. The following shows that the conclusion of \cref{prop:optimal-embed} also no longer holds if we replace the assumption that $F$ is a Frobenius endomorphism with the assumption that $F$ is a Steinberg endomorphism.
\end{pa}

\begin{lem}\label{lem:no-smooth-embedding-suzuki}
Assume $p=2$ and $\bG$ is simple and simply connected of type $\C_2$. We assume $F : \bG \to \bG$ is a Steinberg endomorphism such that $\bG^F$ is of type ${}^2\C_2$. Then there is no smooth regular embedding $\bG \to \bG'$ such that $\dim(Z(\bG')) = 1$.
\end{lem}

\begin{proof}
Assume $\pi : \mathcal{R}' \to \mathcal{R}$ is a derived embedding of root data where $\mathcal{R}' = (X',\Phi',\widecheck{X}',\widecheck{\Phi}')$ and $\Tor(X'/\mathbb{Z}\Phi') = \{0\}$. Hence, $\pi$ is a smooth regular embedding in the absence of $2$-Steinberg endomorphisms. As $\pi(\Phi'^{\top}) = \mathbb{Z}\Phi$ we can assume, by \cref{lem:iso-of-der-embed}, that $\mathcal{R}' = \mathcal{R} \oplus_{(A,h,f)} \mathcal{T}$ where $\mathcal{T} = (T,\emptyset,\widecheck{T},\emptyset)$ is a torus, $A = X/\mathbb{Z}\Phi$, $h : X \to A$ is the natural projection map, and $f : T \to A$ is a surjective homomorphism.

Let us assume that $\rk(\mathcal{T}) = 1$ then $T = \mathbb{Z}\xi$, for some $\xi \in T$, and as $A \cong \mathbb{Z}/2\mathbb{Z}$ we must have $2\xi$ is a basis of $\Ker(f)$. With the notation as in \cref{exmp:Sp4} we have
\begin{equation*}
X' = \{(a\omega_1+b\omega_2,c\xi) \in X \oplus T \mid a-c \in 2\mathbb{Z}\} = \mathbb{Z}e_1 \oplus \mathbb{Z}e_2 \oplus \mathbb{Z}e_3
\end{equation*}
where $e_1 = (\omega_1,\xi)$, $e_2 = (\omega_2,0)$ and $e_3 = (0,2\xi)$. We can assume that $F$ is such that $F^*(\omega_1) = 2^r\omega_2$ and $F^*(\omega_2) = 2^{r+1}\omega_1$ for some integer $r > 0$. We assume for a contradiction that there exists a $2$-Steinberg endomorphism $\psi : \mathcal{R}' \to \mathcal{R}'$ such that $p_1 \circ \psi = F^*\circ p_1$ where $p_1 : \mathcal{R}' \to \mathcal{R}$ is the natural projection map. By \cref{cor:iso-central-prod,rem:lifting-Steinberg-Frob} it follows that there exists an integer $s \geqslant 0$ such that $\psi(x,y) = (F^*(x),2^sy)$ for all $(x,y) \in X' = X \oplus_{(A,h,f)} T$.

As $\psi$ is assumed to be a $2$-Steinberg endomorphism there must be some power of $\psi$ which acts as multiplication by a $2$-power. Clearly no odd power of $\psi$ can have this property because no odd power of $F^*$ has this property. Furthermore, one readily checks that $\psi^2(e_2) = 2^{2r+1}e_2$ and $\psi^2(e_3) = 2^{2s}e_3$ so clearly no even power of $\psi$ can have this property.
\end{proof}

\section{Asai's Reduction Techniques}\label{sec:asai-red-technique}
\begin{proof}[{}{of \cref{lem:completion-of-embedding}}]
Let us fix a maximal torus $\bT \leqslant \bG$ and maximal tori $\bT_i \leqslant \bG_i$ such that $\sigma_i(\bT) \leqslant \bT_i$. If $\mathcal{R} = \mathcal{R}(\bG,\bT) = (X,\Phi,\widecheck{X},\widecheck{\Phi})$ and $\mathcal{R}_i = \mathcal{R}(\bG_i,\bT_i) = (X_i,\Phi_i,\widecheck{X}_i,\widecheck{\Phi}_i)$ are the corresponding root data then $\sigma_i^* : \mathcal{R}_i \to \mathcal{R}$ is a derived embedding of root data by \cref{pa:der-embed-induces-der-embed}. We define a torus $\mathcal{S} = (S,\emptyset,\widecheck{S},\emptyset)$ by setting
\begin{equation*}
S = \{(x_1,x_2) \in X_1 \oplus X_2 \mid \sigma_1^*(x_1) = \sigma_2^*(x_2)\}
\end{equation*}
and $\widecheck{S} = \Hom(S,\mathbb{Z})$.

The torus $\mathcal{S}$ is equipped with a homomorphism $f : S \to A := X/\mathbb{Z}\Phi$ defined by
\begin{equation*}
f(x_1,x_2) = \sigma_1^*(x_1) + \mathbb{Z}\Phi = \sigma_2^*(x_2) + \mathbb{Z}\Phi,
\end{equation*}
which is surjective because the homomorphisms $\sigma_i^* : X_i \to X$ are surjective. In particular, we may form the central product $\mathcal{R} \oplus_{(A,h,f)} \mathcal{S}$ where $h : X \to A$ is the natural projection map. It follows from \cref{lem:fibre-prod} that the quotient $(X \oplus_{(A,h,f)} S)/\mathbb{Z}\Phi$ has no torsion.

By \cref{lem:iso-of-der-embed} there exists an isomorphism $\phi_i : \mathcal{R}_i \to \mathcal{R} \oplus_{(A_i,h_i,f_i)}\mathcal{S}_i$ where $\mathcal{S}_i = (\mathcal{R}_i)_{\rad}$, $A_i = X/\sigma_i^*(\Phi_i^{\top})$, $h_i : X \to A_i$ is the natural projection map, and $f_i : S_i \to A_i$ is given by $f_i(y + \Phi_i^{\top}) = \sigma_i^*(y) + \sigma_i^*(\Phi_i^{\top})$. Moreover, we note that $\lambda_i = \sigma_i^*\circ\phi_i^{-1} : \mathcal{R} \oplus_{(A_i,h_i,f_i)}\mathcal{S}_i \to \mathcal{R}$ is the projection onto the first factor.

Now, let us observe that we have homomorphisms $\tau_i : \mathcal{S} \to \mathcal{S}_i$, defined by $\tau_i(x_1,x_2) = x_i + \Phi_i^{\top}$, which are surjective because the maps $\sigma_i^*$ are surjective. We claim that the maps $\pi_i : \mathcal{R} \oplus_{(A,h,f)}\mathcal{S} \to \mathcal{R} \oplus_{(A_i,h_i,f_i)}\mathcal{S}_i$ defined by $\pi_i(x,y) = (x,\tau_i(y))$ are surjective homomorphisms of root data. For this, assume $(x,y+\Phi_i^{\top}) \in X_i\oplus_{(A_i,h_i,f_i)} (X_i/\Phi_i^{\top})$ then we have $x - \sigma_i^*(y) \in \sigma_i^*(\Phi_i^{\top})$ so there exists an element $m \in \Phi_i^{\top}$ such that $x - \sigma_i^*(y) = \sigma_i^*(m)$. Or, in other words, we have $x = \sigma_i^*(y+m)$. By the surjectivity of $\tau_i : S \to S_i$ there exists an element $s = (s_1,s_2) \in S$ such that $s_i = y+m$ and so $\tau_i(s) = s_i + \Phi_i^{\top} = y+m+\Phi_i^{\top}$. From the definition of $f$ we have
\begin{equation*}
f(s) = \sigma_i^*(s_i) + \mathbb{Z}\Phi = \sigma_i^*(y+m) + \mathbb{Z}\Phi = x + \mathbb{Z}\Phi
\end{equation*}
so $(x,s) \in X \oplus_{(A,h,f)} S$ and $\pi_i(x,s) = (x,\tau_i(s)) = (x,y+m+\Phi_i^{\top}) = (x,y+\Phi_i^{\top})$, which proves the claim.

It is clear that we have a commutative diagram
\begin{equation}\label{eq:comm-diag-rd-homs}
\begin{tikzpicture}[baseline=(current  bounding  box.center)]
\matrix (m) [matrix of math nodes, row sep=3em, column sep=3em, text height=1.5ex, text depth=0.25ex]
{ \mathcal{R} \oplus_{(A,h,f)} \mathcal{S} & \mathcal{R} \oplus_{(A_1,h_1,f_1)}\mathcal{S}_1\\
  \mathcal{R} \oplus_{(A_2,h_2,f_2)}\mathcal{S}_2 & \mathcal{R}\\ };

\path [>=stealth,->]	(m-1-1) edge node[font=\small,above,text height=1.5ex,text depth=0.25ex,->] {$\pi_1$} (m-1-2)
		(m-1-1) edge node[font=\small,left,text height=1.5ex,text depth=0.25ex] {$\pi_2$} (m-2-1)
		(m-1-2) edge node[font=\small,right,text height=1.5ex,text depth=0.25ex] {$\lambda_1$} (m-2-2)
		(m-2-1) edge node[font=\small,below,text height=1.5ex,text depth=0.25ex] {$\lambda_2$} (m-2-2);
\end{tikzpicture}
\end{equation}
Now, appealing to \cref{thm:full-functor,prop:der-embeds-root-data} there exists a pair $(\bG',\bT')$ and smooth regular embeddings $\sigma_i' : (\bG_i,\bT_i) \to (\bG',\bT')$ such that $\mathcal{R}(\bG',\bT') = \mathcal{R} \oplus_{(A,h,f)} \mathcal{S}$ and $\mathcal{R}(\sigma_i') = \phi_i^{-1}\circ\pi_i$. Moreover, as $\lambda_1\circ\pi_1 = \mathcal{R}(\sigma_1'\circ\sigma_1) = \mathcal{R}(\sigma_2'\circ\sigma_2) = \lambda_2\circ\pi_2$ there exists an element $t \in \bT$ such that $\sigma_1\circ\sigma_1' = \sigma_2\circ\sigma_2'\circ\Inn t$. Hence, replacing $\sigma_2'$ by $\sigma_2'\circ\Inn t$ we obtain the result in the absence of Steinberg endomorphisms.

We now assume that $\bG$ is endowed with a Steinberg/Frobenius endomorphism $F : \bG \to \bG$. By assumption there exist Steinberg/Frobenius endomorphisms $F_i : \bG_i \to \bG_i$ such that $\sigma\circ F = \sigma_i\circ F$. We will assume fixed a pinning $(\bG,\bB,\bT,(x_{\alpha}))$ that is compatible with $F$, c.f., \cref{rem:compat-real}. From this pinning we construct a pinning $(\bG_i,\bB_i,\bT_i,(x^{(i)}_{\alpha}))$ by setting
\begin{equation*}
\bT_i = \sigma(\bT)Z(\bT_i), \qquad \bB_i = \sigma(\bB)Z(\bG_i), \qquad x^{(i)}_{\alpha} := \sigma_i\circ x_{\alpha}.
\end{equation*}
As $F_i\circ \sigma_i = \sigma_i \circ F$ one easily sees that the pinning $(\bG_i,\bB_i,\bT_i,(x^{(i)}_{\alpha}))$ is compatible with $F_i$ and we have isotypies
\begin{equation*}
\sigma_i : (\bG,\bB,\bT,(x_{\alpha})) \to (\bG_i,\bB_i,\bT_i,(x^{(i)}_{\alpha})).
\end{equation*}

We wish to now endow $\bG'$ with an appropriate Steinberg/Frobenius endomorphism. With this in mind we define a $\mathbb{Z}$-module homomorphism $\psi : X\oplus S \to X \oplus S$ by setting
\begin{equation*}
\psi(x,(s_1,s_2)) = (F^*(x),(F_1^*(s_1),F_2^*(s_2))),
\end{equation*}
which makes sense because $\sigma_i^*\circ F_i^* = F^*\circ\sigma_i^*$. Moreover it is easily seen that $\psi$ induces a $p$-Steinberg/$p$-Frobenius endomorphism $\mathcal{R} \oplus_{(A,h,f)} \mathcal{S} \to \mathcal{R} \oplus_{(A,h,f)} \mathcal{S}$. Hence, by \cref{thm:full-functor,prop:characterisation-of-Steinberg-Frobenius} there exists a Steinberg/Frobenius endomorphism $F' : (\bG',\bT') \to (\bG',\bT')$ such that $\mathcal{R}(F') = \psi$. Now we have obtained $F'$ we will assume that $(\bG',\bB',\bT',(x_{\alpha}'))$ is a pinning compatible with $F'$.

We claim that we may assume the commutative diagram in \cref{eq:comm-diag-rd-homs} provides a commutative diagram
\begin{equation*}
\begin{tikzpicture}[baseline=(current  bounding  box.center)]
\matrix (m) [matrix of math nodes, row sep=3em, column sep=4.5em, text height=1.5ex, text depth=0.25ex]
{ \widehat{\mathcal{R}}(\bG',\bB',\bT',(x_{\alpha}')) & \widehat{\mathcal{R}}(\bG_1,\bB_1,\bT_1,(x_{\alpha}^{(1)}))\\
  \widehat{\mathcal{R}}(\bG_2,\bB_2,\bT_2,(x_{\alpha}^{(2)})) & \widehat{\mathcal{R}}(\bG,\bB,\bT,(x_{\alpha}))\\ };

\path [>=stealth,->]	(m-1-1) edge node[font=\small,above,text height=1.5ex,text depth=0.25ex,->] {$\phi_1^{-1}\circ\pi_1$} (m-1-2)
		(m-1-1) edge node[font=\small,left,text height=1.5ex,text depth=0.25ex] {$\phi_2^{-1}\circ\pi_2$} (m-2-1)
		(m-1-2) edge node[font=\small,right,text height=1.5ex,text depth=0.25ex] {$\sigma_1^*$} (m-2-2)
		(m-2-1) edge node[font=\small,below,text height=1.5ex,text depth=0.25ex] {$\sigma_2^*$} (m-2-2);
\end{tikzpicture}
\end{equation*}
of $p$-morphisms between based root data. Indeed, if $\bB$ determines the set of simple roots $\Delta \subseteq \Phi \subseteq X(\bT)$, then by the definition of $\psi$ we may certainly assume that $\bB'$ determines the set of simple roots obtained as the image of $\Delta$ embedded in the first factor of $X \oplus S$. With this the claim follows.

Applying \cref{prop:unique-morphism} we obtain unique morphisms
\begin{equation*}
\sigma_i' : (\bG_i,\bB_i,\bT_i,(x_{\alpha}^{(i)})) \to (\bG',\bB',\bT',(x_{\alpha}'))
\end{equation*}
such that $\widehat{\mathcal{R}}(\sigma_i') = \phi_i^{-1}\circ\pi_i$. By the commutativity of the above diagram we have $\widehat{\mathcal{R}}(\sigma_1'\circ\sigma_1) = \widehat{\mathcal{R}}(\sigma_2'\circ\sigma_2)$ hence by \cref{prop:unique-morphism} we have $\sigma_1'\circ\sigma_1 = \sigma_2'\circ\sigma_2$.

The central product $\mathcal{R} \oplus_{(A_i,h_i,f_i)}\mathcal{S}_i$ is endowed with a $p$-Steinberg/$p$-Frobenius endomorphism defined by $\gamma_i(x,s) = (F^*(x),F_i^*(s))$. Moreover, we have $\phi_i \circ F_i^* = \gamma_i\circ\phi_i$ because $\sigma_i^*\circ F_i^* = F^*\circ\sigma_i^*$, c.f., \cref{lem:iso-of-der-embed}. One readily checks that we have $\gamma_i\circ\pi_i = \pi_i\circ\psi$ so
\begin{equation*}
\widehat{\mathcal{R}}(\sigma_i'\circ F_i) = F_i^*\circ(\phi_i^{-1}\circ \pi_i) = \phi_i^{-1}\circ \gamma_i\circ\pi_i = (\phi_i^{-1}\circ \pi_i) \circ \psi = \widehat{\mathcal{R}}(F'\circ\sigma_i').
\end{equation*}
Appealing again to \cref{prop:unique-morphism} we get that $\sigma_i'\circ F_i = F'\circ\sigma_i'$ which completes the proof.
\end{proof}

\begin{proof}[of \cref{prop:quotient-by-sc-der}]
Let $\mathcal{R} = \mathcal{R}(\bG,\bT) = (X,\Phi,\widecheck{X},\widecheck{\Phi})$ be the root datum of $\bG$ with respect to a maximal torus $\bT \leqslant \bG$. By \cref{lem:smooth-reg-embed} there exists a smooth regular embedding $f : \mathcal{R}' \to \widecheck{\mathcal{R}}$ of root data. In particular, if $\mathcal{R}' = (X',\Phi',\widecheck{X}',\widecheck{\Phi}')$ then $X'/\mathbb{Z}\Phi'$ has no torsion. As $f : X' \to \widecheck{X}$ is surjective we have a short exact sequence
\begin{equation*}
\begin{tikzpicture}[baseline={([yshift=-1ex]current bounding box.center)}]
\matrix (m) [matrix of math nodes, row sep=4em, column sep=2em, text height=1.5ex, text depth=0.25ex]
{ 0 & \Ker(f) & X' & \widecheck{X} & 0\\};

\path [>=stealth,->]	(m-1-1) edge (m-1-2)
		(m-1-2) edge (m-1-3)
		(m-1-3) edge (m-1-4)
		(m-1-4) edge (m-1-5);
\end{tikzpicture}
\end{equation*}
which splits because $\widecheck{X}$ is a free $\mathbb{Z}$-module. This implies that the dual $\widecheck{f} : X \to \widecheck{X}'$ is injective and the quotient $\widecheck{X}'/\widecheck{f}(X)$ is free because $\widecheck{f}(X)$ has a complement. By \cref{thm:full-functor} there exists a pair $(\widetilde{\bG},\widetilde{\bT})$, such that $\mathcal{R}(\widetilde{\bG},\widetilde{\bT}) = \widecheck{\mathcal{R}}'$, and an isotypy $\pi : (\widetilde{\bG},\widetilde{\bT}) \to (\bG,\bT)$ such that $\mathcal{R}(\pi) = \widecheck{f} : \mathcal{R} \to \widecheck{\mathcal{R}}'$. According to \cite[V.22.4]{borel:1991:linear-algebraic-groups} and \cite[\S5]{steinberg:1999:the-isomorphism-and-isogeny-theorems} we have $\pi$ is a surjective central homomorphism.

If $\widetilde{\bG}_{\der} \leqslant \widetilde{\bG}$ is the derived subgroup of $\widetilde{\bG}$ then $\widetilde{\bT}_{\der} = \widetilde{\bT} \cap \widetilde{\bG}_{\der} \leqslant \widetilde{\bG}_{\der}$ is a maximal torus. According to \cite[8.1.9]{springer:2009:linear-algebraic-groups} we have $\mathcal{R}(\widetilde{\bG}_{\der},\widetilde{\bT}_{\der}) = \widecheck{\mathcal{R}}_{\der}' = (\widecheck{X}'/\widecheck{\Phi}'^{\perp},\widecheck{\Phi}',\Phi'^{\top},\Phi')$. This implies $\widetilde{\bG}_{\der}$ is simply connected because $\Tor(X'/\mathbb{Z}\Phi') = \{0\}$ so $\Phi'^{\top} = \mathbb{Z}\Phi'$.

Recall that $\Ker(\pi)$ is contained in $Z(\widetilde{\bG})$ so is a diagonalisable group. If $\iota : \Ker(\pi) \to \widetilde{\bT}$ is the natural closed embedding then we have a short exact sequence of abelian groups
\begin{equation*}
\begin{tikzpicture}[baseline={([yshift=-1ex]current bounding box.center)}]
\matrix (m) [matrix of math nodes, row sep=4em, column sep=2em, text height=1.5ex, text depth=0.25ex]
{ 0 & X(\bT) & X(\widetilde{\bT}) & X(\Ker(\pi)) & 0\\};

\path [>=stealth,->]	(m-1-1) edge (m-1-2)
		(m-1-2) edge node [above] {$\pi^*$} (m-1-3)
		(m-1-3) edge node [above] {$\iota^*$} (m-1-4)
		(m-1-4) edge (m-1-5);
\end{tikzpicture}
\end{equation*}
because $X(-)$ is exact, c.f., \cite[0.21]{digne-michel:1991:representations-of-finite-groups-of-lie-type}. In particular, we have $X(\Ker(\pi)) \cong \widecheck{X}'/\widecheck{f}(X)$ but, as was noted above, the quotient $\widecheck{X}'/\widecheck{f}(X)$ has no torsion so $\Ker(\pi)$ is a torus.

Now $Z(\widetilde{\bG})$ is connected if and only if $\Tor_{p'}(\widecheck{X}'/\mathbb{Z}\widecheck{\Phi}') = \{0\}$. As $\widecheck{f}(X)$ has a complement we have $\Tor(\widecheck{X}'/\mathbb{Z}\widecheck{\Phi}') = \Tor(\widecheck{f}(X)/\mathbb{Z}\widecheck{\Phi}')$. Moreover, as $\widecheck{f}(X)/\mathbb{Z}\widecheck{\Phi}' = \widecheck{f}(X/\mathbb{Z}\Phi)$ and $\widecheck{f}$ is injective we have $\Tor(\widecheck{X}'/\mathbb{Z}\widecheck{\Phi}') \cong \Tor(X/\mathbb{Z}\Phi)$ so $Z(\widetilde{\bG})$ is connected/smooth if and only if $Z(\bG)$ is connected/smooth.

Finally, assume $\bG$ is endowed with a Steinberg/Frobenius endomorphism $F : \bG \to \bG$ then we may assume $\bT \leqslant \bG$ is $F$-stable so that $\mathcal{R}(F) = F^* : \mathcal{R} \to \mathcal{R}$ is the corresponding $p$-Steinberg/$p$-Frobenius endomorphism. By \cref{pa:dual-p-morphism} we see that the dual morphism $\widecheck{F}^* : \widecheck{\mathcal{R}} \to \widecheck{\mathcal{R}}$ is again a $p$-Steinberg/$p$-Frobenius endomorphism. As $f : \mathcal{R}' \to \widecheck{\mathcal{R}}$ is a smooth embedding of root data there exists a $p$-Steinberg/$p$-Frobenius endomorphism $\phi : \mathcal{R}' \to \mathcal{R}'$ such that $\widecheck{F}^*\circ f = f\circ\phi$. By \cref{pa:dual-p-morphism} we have duality is bijective and contravariant on $p$-morphisms which implies that $\widecheck{f}\circ F^* = \widecheck{\phi}\circ\widecheck{f}$. Again appealing to \cref{pa:dual-p-morphism} we have $\widecheck{\phi} : \widecheck{\mathcal{R}}' \to \widecheck{\mathcal{R}}'$ is a $p$-Steinberg/$p$-Frobenius endomorphism. By \cref{lem:lift-frob} there exists a Steinberg/Frobenius endomorphism $\widetilde{F} : (\widetilde{\bG},\widetilde{\bT}) \to (\widetilde{\bG},\widetilde{\bT})$ such that $\mathcal{R}(\widetilde{F}) = \widecheck{\phi}$ and $\pi\circ \widetilde{F} = F\circ \pi$, which completes the proof.
\end{proof}

\section{Cyclically Permuted Factors}\label{sec:cyc-perm-factors}
\begin{assumption}
From this section on we assume $\mathbb{K} = \overline{\mathbb{F}_p}$ and $p>0$ is a prime. We choose an algebraic closure $\Ql$ with $\ell\neq p$ a prime and fix an involutive automorphism $\overline{\phantom{x}} : \Ql \to \Ql$ which maps every root of unity to its inverse. For any finite group $H$ we denote by $\Class(H)$ the vector space of all $\Ql$-class functions $f : H \to \Ql$. We consider this to be an inner product space with respect to the usual form defined by $\langle f,f'\rangle_H := |H|^{-1}\sum_{h\in H}f(h)\overline{f'(h)}$. The $\Ql$-irreducible characters are denoted by $\Irr(H) \subseteq \Class(H)$.
\end{assumption}

\begin{pa}\label{pa:permuted-factors-2}
In this section, we assume we are in the setting of \cref{pa:permuted-factors}. An easy calculation shows that any $F$-stable subset $X \subseteq \bG$ is of the form $X_1 \times F(X_1) \times \cdots \times F^{n-1}(X_1)$ for some $F^n$-stable subset $X_1 \subseteq \bG_1$. Hence the projection map $\pi_1 : \bG \to \bG_1$ clearly induces a bijection between the $F$-stable subsets of $\bG$ and the $F^n$-stable subsets of $\bG_1$.
\end{pa}

\begin{proof}[of \cref{lem:cyc-perm-smooth-embed}]
Let $\bT = \bT_1 \times \cdots \times \bT_n$ be an $F$-stable maximal torus of $\bG$ then $F(\bT_i) = \bT_{i+1}$. If $\mathcal{R} = \mathcal{R}(\bG,\bT) = (X,\Phi,\widecheck{X},\widecheck{\Phi})$ is the root datum of $(\bG,\bT)$ then we have $\mathcal{R} = \mathcal{R}_1 \oplus \cdots \oplus \mathcal{R}_n$ where $\mathcal{R}_i = \mathcal{R}(\bG_i,\bT_i) = (X_i,\Phi_i,\widecheck{X}_i,\widecheck{\Phi}_i)$ is the root datum of $(\bG_i,\bT_i)$. Let $f_i : X_i \to A_i := X_i/\mathbb{Z}\Phi_i$ be the natural projection map then we have smooth regular embeddings $h_i : \mathcal{R}_i' \to \mathcal{R}_i$ where $\mathcal{R}_i' = \mathcal{R}_i \oplus_{(A_i,f_i,f_i)} \mathcal{R}_i^{\circ} = (X_i',\Phi_i',\widecheck{X}_i',\widecheck{\Phi}_i')$, see \cref{lem:smooth-reg-embed}. If $\mathcal{R}' = (X',\Phi',\widecheck{X}',\widecheck{\Phi}')$ is the root datum $\mathcal{R}_1' \oplus \cdots \oplus\mathcal{R}_n'$ then we have a natural surjective homomorphism of root data $h = h_1\oplus \cdots \oplus h_n : \mathcal{R}' \to \mathcal{R}$.

Now $X$ is a direct sum $X_1 \oplus \cdots \oplus X_n$ and by the assumption on $\bT$ we have $F^*(X_{i+1}) = X_i$. We certainly have a $\mathbb{Z}$-module homomorphism $\psi_i : X_{i+1}\oplus X_{i+1} \to X_i\oplus X_i$ defined by $\psi_i(x,y) = (F^*(x),F^*(y))$. From the definition it is readily checked that this restricts to a $\mathbb{Z}$-module homomorphism $\psi_i : X_{i+1}' \to X_i'$. Clearly the $\mathbb{Z}$-module $X'$ is a direct sum $X_1' \oplus \cdots \oplus X_n'$ and we may define a $\mathbb{Z}$-module homomorphism $\psi : X' \to X'$ by setting $\psi(x_1,\dots,x_n) = (\psi_1(x_2),\dots,\psi_{n-1}(x_n),\psi_n(x_1))$. It is easily checked that $\psi$ induces a $p$-Steinberg endomorphism $\mathcal{R}' \to \mathcal{R}'$ such that $h\circ\psi = F^*\circ h$. Moreover $\psi^n$ stabilises $\mathcal{R}_1'$ and satisfies $h_1\circ\psi^n = {F^*}^n \circ h$. From this the statement of the lemma is easily obtained, as in the proof of \cref{lem:smooth-reg-embed}. We leave the details to the reader.
\end{proof}

\subsection{Deligne--Lusztig Induction and Restriction}
\begin{pa}\label{pa:permuted-factors-par-levi}
An easy calculation shows that the projection map $\pi_1 : \bG \to \bG_1$ restricts to an isomorphism of finite groups $\bG^F \to \bG_1^{F^n}$, so the inflation $\pi_1^* : \Class(\bG_1^{F^n}) \to \Class(\bG^F)$ through $\pi_1$ is an isometry. It is our purpose now to show that the construction of Deligne--Lusztig induction is compatible with $\pi_1$. For this let us assume that $\bP_1 \leqslant \bG_1$ is a parabolic subgroup of $\bG_1$ with $\bL_1 \leqslant \bP_1$ an $F^n$-stable Levi complement of $\bP_1$. There is then a unique $F$-stable subgroup $\bL \leqslant \bG$ such that $\pi_1(\bL) = \bL_1$, which is of the form $\bL_1 \times F(\bL_1) \times \cdots \times F^{n-1}(\bL_1)$, c.f., \cref{pa:permuted-factors-2}. As the parabolic subgroup $\bP_1$ is not necessarily $F$-stable there may be many choices of parabolic subgroup $\bP \leqslant \bG$ such that $\pi_1(\bP) = \bP_1$. We will assume that $\bP = \bP_1 \times F^{-n+1}(\bP_1) \times \cdots \times F^{-1}(\bP_1)$; this makes sense because $F^n(\bP_1)$ is also a parabolic subgroup of $\bG_1$ with $\bL_1$ as an $F^n$-stable Levi complement. With this we may form the Deligne--Lusztig induction maps $R_{\bL\subseteq\bP}^{\bG} : \Class(\bL^F) \to \Class(\bG^F)$ and $R_{\bL_1\subseteq\bP_1}^{\bG_1} : \Class(\bL_1^{F^n}) \to \Class(\bG_1^{F^n})$, c.f., \cite[1]{lusztig:1976:on-the-finiteness}. The following is an analogue of \cite[13.22]{digne-michel:1991:representations-of-finite-groups-of-lie-type} in our setting.
\end{pa}

\begin{prop}\label{prop:comm-DL-ind}
Assume the notation and assumptions of \cref{pa:permuted-factors,pa:permuted-factors-par-levi} then we have a commutative diagram
\begin{center}
\begin{tikzpicture}
\matrix (m) [matrix of math nodes, row sep=3em, column sep=3em, text height=1.5ex, text depth=0.25ex]
{ \Class(\bL_1^{F^n}) & \Class(\bG_1^{F^n})\\
  \Class(\bL^F) & \Class(\bG^F)\\ };

\path [>=stealth,->]	(m-1-1) edge node[font=\small,above,text height=1.5ex,text depth=1ex,->] {$R_{\bL_1\subseteq \bP_1}^{\bG_1}$} (m-1-2)
		(m-1-1) edge node[font=\small,left,text height=1.5ex,text depth=0.25ex] {$\pi_1^*$} (m-2-1)
		(m-1-2) edge node[font=\small,right,text height=1.5ex,text depth=0.25ex] {$\pi_1^*$} (m-2-2)
		(m-2-1) edge node[font=\small,below,text height=1.5ex,text depth=0.25ex] {$R_{\bL\subseteq\bP}^{\bG}$} (m-2-2);
\end{tikzpicture}
\end{center}
\end{prop}

\begin{proof}
For any $\mathbb{K}$-variety $\bX$ and any finite order automorphism $h \in \Aut(\bX)$ we define the Lefschetz trace
\begin{equation*}
\mathscr{L}(h \mid \bX) = \sum_{i \in \mathbb{Z}} (-1)^i\Tr(h \mid H_c^i(\bX,\Ql)),
\end{equation*}
where $H_c^i(\bX,\Ql)$ is the $i$th compactly supported $\ell$-adic cohomology group of $\bX$. Let $\bU \leqslant \bP$ be the unipotent radical of $\bP$ then $\bU_1 = \pi(\bU)$ is clearly the unipotent radical of $\bP_1$ and we have $\bU = \bU_1 \times F^{-n+1}(\bU_1) \times \cdots \times F^{-1}(\bU_1)$. Following \cite[1]{lusztig:1976:on-the-finiteness} we define varieties
\begin{align*}
\bY_{\bU}^{\bG} &= \{x \in \bG \mid x^{-1}F(x) \in \bU\},\\
\bY_{\bU_1}^{\bG_1} &= \{x_1 \in \bG_1 \mid x_1^{-1}F^n(x_1) \in \bU_1\},
\end{align*}
which are endowed with natural $\bG^F \times (\bL^F)^{\opp}$ and $\bG_1^{F^n} \times (\bL_1^{F^n})^{\opp}$ actions defined by left and right translation. By \cite[Proposition 4.5]{digne-michel:1991:representations-of-finite-groups-of-lie-type} we have for any $\chi \in \Class(\bL_1^{F^n})$ and $g \in \bG^F$ that
\begin{align*}
(\pi_1^* \circ R_{\bL_1\subseteq\bP_1}^{\bG_1})(\chi)(g) &= |\bL_1^{F^n}|^{-1}\sum_{l_1 \in \bL_1^{F^n}} \mathscr{L}((\pi_1(g),l_1) \mid \bY_{\bU_1}^{\bG_1})\chi(l_1^{-1}),\\
(R_{\bL_1\subseteq\bP_1}^{\bG_1}\circ \pi_1^*)(\chi)(g) &= |\bL^F|^{-1}\sum_{l \in \bL^F} \mathscr{L}((g,l) \mid \bY_{\bU}^{\bG})\chi(\pi_1(l)^{-1}).
\end{align*}
As $\pi_1$ induces isomorphisms $\bG^F \to \bG_1^{F^n}$ and $\bL^F \to \bL_1^{F^n}$ we see immediately from these formulas that we need only show that
\begin{equation}\label{eq:Lefschetz-trace}
\mathscr{L}((g,l) \mid \bY_{\bU}^{\bG}) = \mathscr{L}((\pi_1(g),\pi_1(l)) \mid \bY_{\bU_1}^{\bG_1})
\end{equation}
for any $(g,l) \in \bG^F \times (\bL^F)^{\opp}$.

From the definition of $\bY_{\bU}^{\bG}$ we see that any element $x \in \bY_{\bU}^{\bG}$ is of the form $(x_1,\dots,x_n)$ with $x_i \in \bG_i$ and
\begin{equation}\label{eq:condition}
x_{i+1}^{-1}F(x_i) \in F^{-n+i}(\bU_1),
\end{equation}
with the indices computed cyclically. From this we see that
\begin{equation*}
x_1^{-1}F^n(x_1) = x_1^{-1}F(x_n)\cdot F(x_n^{-1}F(x_{n-1})) \cdot F^2(x_{n-1}^{-1}F(x_{n-2})) \cdots F^{n-1}(x_2^{-1}F(x_1)) \in \bU_1.
\end{equation*}
In particular we have $\pi$ defines a surjective morphism $\bY_{\bU}^{\bG} \to \bY_{\bU_1}^{\bG_1}$. The surjectivity is easy to see because for any $x_1 \in \bY_{\bU_1}^{\bG_1}$ we have $\hat{x}_1 = (x_1,F(x_1),\dots,F^{n-1}(x_1)) \in \bG$ satisfies
\begin{equation*}
\hat{x}_1^{-1}F(\hat{x}_1) = (x_1^{-1}F^n(x_1),1,\dots,1),
\end{equation*}
so is clearly contained in $\bY_{\bU}^{\bG}$. Now assume $u = (1,u_1,\dots,u_{n-1}) \in \bG$ is such that $u_i \in F^{-n+i}(\bU_1)$ then we have
\begin{equation*}
(\hat{x}_1u)^{-1}F(\hat{x}_1u) = (x_1^{-1}F^n(x_1)F(u_{n-1}),u_1^{-1},u_2^{-1}F(u_1),u_3^{-1}F(u_2),\dots,u_{n-1}^{-1}F(u_{n-2}))
\end{equation*}
which is certainly contained in $\bU$. This shows that $\hat{x}_1u \in \bY_{\bU}^{\bG}$ is in the fibre $\pi_1^{-1}(x_1)$ of $x_1 \in \bY_{\bU_1}^{\bG_1}$. Moreover the condition in \cref{eq:condition} shows that every element of $\pi_1^{-1}(x_1)$ is of this form; so $\pi_1^{-1}(x_1)$ is affine of dimension $(n-1)\dim \bU_1$. The desired equality in \cref{eq:Lefschetz-trace} now follows from \cite[10.12(ii)]{digne-michel:1991:representations-of-finite-groups-of-lie-type}.
\end{proof}

\subsection{Lusztig Series}
\begin{pa}
Fix an $F$-stable maximal torus $\bT_0 \leqslant \bG$ and let $\mathcal{R} = \mathcal{R}(\bG,\bT_0)$ be the corresponding root datum. By \cref{thm:full-functor} there exists a pair $(\bG^{\star},\bT_0^{\star})$, unique up to isomorphism, such that $\mathcal{R}(\bG^{\star},\bT_0^{\star}) = \widecheck{\mathcal{R}}$. By the above discussion the torus $\bT_0$ is of the form $\bT_1 \times \cdots \times \bT_n$ where $\bT_i = F^{i-1}(\bT_1)$. We then have $\mathcal{R} = \mathcal{R}_1 \oplus \cdots \oplus \mathcal{R}_n$ where $\mathcal{R}_i = \mathcal{R}(\bG_i,\bT_i)$. Clearly $\widecheck{\mathcal{R}} = \widecheck{\mathcal{R}}_1 \oplus \cdots \oplus \widecheck{\mathcal{R}}_n$ so we must have $\bG^{\star} = \bG_1^{\star} \times \cdots \times \bG_n^{\star}$ and $\bT_0^{\star} = \bT_1^{\star} \times \cdots \times \bT_n^{\star}$ where $\bT_i^{\star} \leqslant \bG_i^{\star}$ is a maximal torus and $\mathcal{R}(\bG_i^{\star},\bT_i^{\star}) = \widecheck{\mathcal{R}}_i$. Finally, again appealing to \cref{thm:full-functor}, there exists a Steinberg endomorphism $F^{\star} : (\bG^{\star},\bT_0^{\star}) \to (\bG^{\star},\bT_0^{\star})$ such that $\mathcal{R}(F^{\star}) = \widecheck{\mathcal{R}(F)}$. By construction we see that $\bT_i^{\star} = (F^{\star})^{i-1}(\bT_1^{\star})$.
\end{pa}

\begin{pa}
As above we have a projection map $\pi_1^{\star} : \bG^{\star} \to \bG_1^{\star}$ which restricts to an isomorphism $\bG^{\star F^{\star}} \to \bG_1^{\star F^{\star n}}$ of finite groups. In particular, $\pi_1^{\star}$ induces a bijection between the conjugacy classes of $\bG^{\star F^{\star}}$ and those of $\bG_1^{\star F^{\star n}}$. Recall that if $[s] \subseteq \bG^{\star F^{\star}}$ is the $\bG^{\star F^{\star}}$-conjugacy class of a semisimple element $s \in \bG^{\star F^{\star}}$ then we have a corresponding rational Lusztig series $\mathcal{E}(\bG^F,[s]) \subseteq \Irr(\bG^F)$. Moreover, we have a decomposition
\begin{equation*}
\Irr(\bG^F) = \bigsqcup_{[s] \subseteq \bG^{\star F^{\star}}} \mathcal{E}(\bG^F,[s])
\end{equation*}
where the union runs over all semisimple conjugacy classes. Similarly we have a decomposition of $\Irr(\bG_1^{F^n})$ into Lusztig series. The following shows these decompositions are compatible with the isometry $\pi_1^* : \Class(\bG_1^{F^n}) \to \Class(\bG^F)$.
\end{pa}

\begin{cor}\label{cor:L-series-cyc-perm}
Let $s \in \bG^{\star F^{\star}}$ be a semisimple element then the isometry $\pi_1^* : \Class(\bG_1^{F^n}) \to \Class(\bG^F)$ restricts to a bijection $\mathcal{E}(\bG_1^{F^n},[s_1]) \to \mathcal{E}(\bG^F,[s])$ where $s_1 = \pi_1^{\star}(s)$.
\end{cor}

\begin{proof}
Let $\mathcal{C}(\bG,F)$ denote the set of all pairs $(\bS,\theta)$ consisting of an $F$-stable maximal torus $\bS \leqslant \bG$ and an irreducible character $\theta \in \Irr(\bS^F)$. Note we have a bijection $\mathcal{C}(\bG_1,F^n) \to \mathcal{C}(\bG,F)$ defined by $(\bS_1,\theta_1) \mapsto (\bS,\theta) = (\bS_1F(\bS_1)\cdots F^{n-1}(\bS_1),\pi_1^*(\theta_1))$. Moreover, by \cref{prop:comm-DL-ind} we have $R_{\bS}^{\bG}(\pi_1^*(\theta_1)) = \pi_1^*(R_{\bS_1}^{\bG}(\theta_1))$. So $\pi_1^*$ maps the irreducible constituents of $R_{\bS_1}^{\bG}(\theta_1)$ onto those of $R_{\bS}^{\bG}(\theta)$.

Now let $\mathcal{S}(\bG^{\star},F^{\star})$ denote the set of all pairs $(\bS^{\star},s)$ consisting of an $F^{\star}$-stable maximal torus $\bS^{\star} \leqslant \bG^{\star}$ and a semisimple element $s \in \bS^{\star F^{\star}}$. Again we have a bijection $\mathcal{S}(\bG_1^{\star},F^{\star n}) \to \mathcal{S}(\bG^{\star},F^{\star})$ defined by $(\bS_1^{\star},s_1) \mapsto (\bS^{\star},s) =  (\bS_1^{\star}F^{\star}(\bS_1^{\star})\cdots (F^{\star})^{n-1}(\bS_1^{\star}),s_1F^{\star}(s_1)\cdots(F^{\star})^{n-1}(s_1))$. The statement follows once we know we have a commutative diagram
\begin{center}
\begin{tikzpicture}
\matrix (m) [matrix of math nodes, row sep=3em, column sep=3em, text height=1ex, text depth=0.25ex]
{ \mathcal{C}(\bG_1,F^n)/\bG_1^{F^n} & \mathcal{C}(\bG,F)/\bG^F\\
  \mathcal{S}(\bG_1^{\star},F^{\star n})/\bG_1^{\star F^{\star n}} & \mathcal{S}(\bG^{\star},F^{\star})/\bG^{\star F^{\star}}\\ };

\path [>=stealth,->]	(m-1-1) edge node[font=\small,above,text height=1.5ex,text depth=1ex,->] {} (m-1-2)
		(m-1-1) edge node[font=\small,left,text height=1.5ex,text depth=0.25ex] {} (m-2-1)
		(m-1-2) edge node[font=\small,right,text height=1.5ex,text depth=0.25ex] {} (m-2-2)
		(m-2-1) edge node[font=\small,below,text height=1.5ex,text depth=0.25ex] {} (m-2-2);
\end{tikzpicture}
\end{center}
between the orbits of the natural conjugation actions of the respective groups. Here the vertical maps are given by the bijection described in \cite[13.13]{digne-michel:1991:representations-of-finite-groups-of-lie-type}, see also \cite[6.7]{taylor:2016:action-of-automorphisms-symplectic}. We leave it as an exercise to the reader to verify the commutativity of this diagram.
\end{proof}

\section{Unipotent Supports}\label{sec:unip-supp}
\begin{pa}
For this section we place ourselves in the setup of \cref{pa:uni-supp-setup}. In what follows we will denote by $\bG_{\uni} = \{g_{\uni} \mid g \in \bG\} \subseteq \bG$ the closed subset consisting of unipotent elements. We note that $\bG_{\uni}$ is $F$-stable.
\end{pa}

\begin{thm}\label{thm:unip-supp}
Assume $p$ is a good prime for $\bG$ and $Z(\bG)$ is connected. If $\chi \in \Irr(\bG^F)$ is an irreducible character and $g \in \bG^F$ is such that $\chi(g) \neq 0$ then $g_{\uni} \in \overline{\mathcal{O}_{\chi}}$.
\end{thm}

\begin{lem}\label{lem:uni-supp-embed}
Assume $Z(\bG)$ is connected and $\iota : \bG \to \widetilde{\bG}$ is a regular embedding then \cref{thm:unip-supp} holds for all irreducible characters of $\bG^F$ if and only if it holds for all irreducible characters of $\widetilde{\bG}^F$.
\end{lem}

\begin{proof}
As remarked in \cite[1.7.6(c)]{geck-malle:2016:reductive-groups-and-steinberg-maps} we have $\widetilde{\bG}^F = Z(\widetilde{\bG})^F\cdot \iota(\bG)^F$. This implies that any class function on $\iota(\bG)^F$ is invariant under conjugation by $\widetilde{\bG}^F$. Hence, if $\widetilde{\chi} \in \Irr(\widetilde{\bG}^F)$ is an irreducible character and $\chi = \widetilde{\chi}\circ\iota$ is the restriction to $\bG^F$ then Clifford's Theorem implies that $\chi' = \frac{1}{e}\chi \in \Irr(\bG^F)$ for some integer $e\geqslant 1$.

Now $\iota$ restricts to an isomorphism of varieties $\bG_{\uni} \to \widetilde{\bG}_{\uni}$ and induces a bijection between the unipotent conjugacy classes of $\bG$ and $\widetilde{\bG}$. If $\widetilde{g} \in \widetilde{\bG}^F$ then we may write this as $\iota(g)z$ with $g \in \bG^F$ and $z\in Z(\widetilde{\bG}^F) = Z(\widetilde{\bG})^F$. As $z$ is central we have $\widetilde{\chi}(\widetilde{g}) = \omega_{\widetilde{\chi}}(z)\widetilde{\chi}(\iota(g))$ where $\omega_{\widetilde{\chi}}(z) = \widetilde{\chi}(z)/\widetilde{\chi}(1) \neq 0$. Hence, we have $\widetilde{\chi}(\widetilde{g}) \neq 0$ if and only if $\chi'(g) \neq 0$. Now, clearly $\widetilde{g}_{\uni} = \iota(g_{\uni})$ and as $\mathcal{O}_{\widetilde{\chi}} = \iota(\mathcal{O}_{\chi'})$, see the proof of \cite[5.1]{geck:1996:on-the-average-values}, the statement follows.
\end{proof}

\begin{lem}\label{lem:uni-supp-smooth-cover}
Assume $Z(\bG)$ is connected and let $\pi : \widetilde{\bG} \to \bG$ be a smooth covering then \cref{thm:unip-supp} holds for all irreducible characters of $\bG^F$ if it holds for all irreducible characters of $\widetilde{\bG}^F$.
\end{lem}

\begin{proof}
Note that $\pi$ restricts to a bijection $\widetilde{\bG}_{\uni} \to \bG_{\uni}$ and induces a bijection between the unipotent conjugacy classes of $\widetilde{\bG}$ and those of $\bG$ because $\Ker(\pi)$ is a central torus. As $\Ker(\pi)$ is connected an easy application of the Lang--Steinberg theorem shows that $\pi$ restricts to a surjective homomorphism $\pi : \widetilde{\bG}^F \to \bG^F$. Hence, given an irreducible character $\chi \in \Irr(\bG^F)$ we have the inflation $\widetilde{\chi} = \chi\circ\pi \in \Irr(\widetilde{\bG}^F)$ is also irreducible. Now assume $\widetilde{g} \in \widetilde{\bG}^F$ then by definition we have $\widetilde{\chi}(\widetilde{g}) = \chi(g)$ where $g = \pi(\widetilde{g}) \in \bG^F$ so $\widetilde{\chi}(\widetilde{g}) \neq 0$ if and only if $\chi(g) \neq 0$. From the proof of \cite[5.2]{geck:1996:on-the-average-values} we see that $\pi(\mathcal{O}_{\widetilde{\chi}}) = \mathcal{O}_{\chi}$ so the conclusion of \cref{thm:unip-supp} holds for $\chi$ if and only if it holds for $\widetilde{\chi}$.
\end{proof}

\begin{proof}[of \cref{thm:unip-supp}]
By \cref{prop:quotient-by-sc-der} there exists a smooth covering $\pi : \widetilde{\bG} \to \bG$ and by \cref{lem:uni-supp-smooth-cover} the statement holds in $\bG$ if it holds in $\widetilde{\bG}$. Hence, we can assume that $\bG$ has a simply connected derived subgroup. As the derived subgroup is simply connected we have $\bG_{\der} = \bG_{\der}^{(1)} \times \cdots \times \bG_{\der}^{(r)}$ where each $\bG_{\der}^{(i)}$ is a direct product of simple groups transitively permuted by $F$. Now assume chosen a regular embedding $\bG_{\der}^{(i)} \hookrightarrow \widetilde{\bG}^{(i)}$ then taking $\widetilde{\bG} = \widetilde{\bG}^{(1)} \times \cdots \times \widetilde{\bG}^{(r)}$ we get a regular embedding $\bG_{\der} \hookrightarrow \widetilde{\bG}$. By \cref{lem:completion-of-embedding} we obtain a commutative diagram
\begin{equation*}
\begin{tikzpicture}
\matrix (m) [matrix of math nodes, row sep=3em, column sep=3em, text height=1.5ex, text depth=0.25ex]
{ \bG_{\der} & \bG\\
  \widetilde{\bG} & \bG'\\ };

\path [>=stealth,->]	(m-1-1) edge node[font=\small,above,text height=1.5ex,text depth=0.25ex,->] {} (m-1-2)
		(m-1-1) edge node[font=\small,left,text height=1.5ex,text depth=0.25ex] {} (m-2-1)
		(m-1-2) edge node[font=\small,right,text height=1.5ex,text depth=0.25ex] {$\sigma_1$} (m-2-2)
		(m-2-1) edge node[font=\small,below,text height=1.5ex,text depth=0.25ex] {$\sigma_2$} (m-2-2);
\end{tikzpicture}
\end{equation*}
where $\sigma_1 : \bG \to \bG'$ and $\sigma_2 : \widetilde{\bG} \to \bG'$ are smooth regular embeddings. Applying \cref{lem:uni-supp-embed} twice we see that the desired statement holds for $\bG$ if and only if it holds for $\widetilde{\bG}$. As the statement is clearly compatible with respect to direct products it thus suffices to prove the statement for each $\widetilde{\bG}^{(i)}$.

Note that we have free reign when choosing the regular embedding $\bG^{(i)} \hookrightarrow \widetilde{\bG}^{(i)}$. Hence, by \cref{lem:cyc-perm-smooth-embed} we may assume that $\bG$ is a direct product $\bG_1 \times \cdots \times \bG_n$ where $(\bG_i)_{\der}$ is simple and simply connected and $F$ is such that $F(\bG_i) = \bG_{i+1}$. Now assume $\chi_1 \in \Irr(\bG_1^{F^n})$ and let $\chi = \chi_1\circ\pi_1$ be the inflation through the projection map $\pi_1 : \bG \to \bG_1$. Note that $\pi_1$ restricts to a bijection $\bG_{\uni}^F \to (\bG_1)_{\uni}^{F^n}$ and induces a bijection between the $F$-stable unipotent conjugacy class of $\bG$ and the $F^n$-stable unipotent conjugacy classes of $\bG_1$. It's clear that $\pi_1(\mathcal{O}_{\chi}) = \mathcal{O}_{\chi_1}$ and the statement holds for $\chi_1$ if and only if it holds for $\chi$. With this we can assume that $\bG$ has a simple and simply connected derived subgroup.

Applying the same trick as above, i.e., \cref{lem:completion-of-embedding,lem:uni-supp-embed}, we may assume that if $\bG_{\der} \cong \SL_n(\mathbb{K})$ then $\bG = \GL_n(\mathbb{K})$. Moreover, as $p$ is a good prime for $\bG$ we have $F$ must be a Frobenius endomorphism. With this assumption we have all the results of \cite{taylor:2016:GGGRs-small-characteristics} are available to us. In particular, we may freely apply all the results of \cite{lusztig:1992:a-unipotent-support}.

%
%

We can now proceed to mimic the proof of \cite[11.2(iv)]{lusztig:1992:a-unipotent-support}. Let $\chi \in \Irr(\bG^F)$ be an irreducible character and $g \in \bG^F$ an element such that $\chi(g) \neq 0$. We set $u = g_{\uni} \in \bG_{\uni}^F$ and denote by $\mathcal{O}_u$ the $\bG$-conjugacy class containing $u$. It suffices to prove the statement assuming that $\mathcal{O}_u$ satisfies the following property. If $\mathcal{O} \neq \mathcal{O}_u$ is a unipotent conjugacy class such that $\mathcal{O}_u \subseteq \overline{\mathcal{O}}$ then we have $\chi(h) = 0$ for any element $h \in \bG^F$ such that $h_{\uni} \in \mathcal{O}$.

This is precisely the assumption made in \cite[9.1]{lusztig:1992:a-unipotent-support} so we may apply \cite[9.2]{lusztig:1992:a-unipotent-support} to deduce that $\chi|_{\mathcal{O}_u^F} \neq 0$. Now, to each unipotent element $v \in \bG_{\uni}^F$ we have a corresponding generalised Gelfand--Graev representation $\Gamma_v$ of $\bG^F$. By \cite[9.10]{lusztig:1992:a-unipotent-support} there exists an element $v\in \mathcal{O}_u^F$ such that $\langle \Gamma_v, \chi^*\rangle \neq 0$, where $\chi^* = \pm D_{\bG^F}(\chi) \in \Irr(\bG^F)$ is the Alvis--Curtis dual of $\chi$. A result of Achar--Aubert thus implies that $\mathcal{O}_u \subseteq \overline{\mathcal{O}_{\chi}}$ which completes the proof, see \cite[Th\'eor\`eme 9.1]{achar-aubert:2007:supports-unipotents-de-faisceaux} and \cite[14.15, 15.2]{taylor:2016:GGGRs-small-characteristics}.
\end{proof}

\setstretch{0.96}
\renewcommand*{\bibfont}{\small}
\printbibliography
\end{document}